\documentclass{amsart}
\usepackage[T1]{fontenc}   

\theoremstyle{definition}

\numberwithin{equation}{section}

\usepackage{amssymb}
\usepackage{amsthm}
\usepackage{amsmath}
\usepackage[mathscr]{eucal}

\usepackage{enumitem}

\theoremstyle{plain}
\newtheorem{thm}{Theorem}[section]		
\newtheorem{prop}[thm]{Proposition}
\newtheorem{cor}[thm]{Corollary}
\newtheorem{lem}[thm]{Lemma}

\theoremstyle{definition}
\newtheorem{df}{Definition}[section]

\theoremstyle{remark}
\newtheorem{rmk}{Remark}[section]
\newtheorem*{ac}{Acknowledgements}

\newcommand{\zz}{\mathbb{Z}}
\newcommand{\qq}{\mathbb{Q}}
\newcommand{\rr}{\mathbb{R}}

\DeclareMathOperator{\card}{Card}

\DeclareMathOperator{\met}{Met}

\newcommand{\yosr}{\mathrm{SR}}

\newcommand{\yorrr}{\mathrm{R}}

\DeclareMathOperator{\yodiam}{diam}

\DeclareMathOperator{\metdis}{\mathcal{D}}

\newcommand{\yobset}{\mathbb{A}}

\newcommand{\yocset}{\mathbb{F}}

\newcommand{\yoinpq}[2]{\Sigma_{\yoqqq}[ #1, #2 ]}

\newcommand{\yofastset}{\mathcal{Q}}

\newcommand{\yoyi}{\Omega}

\newcommand{\yocontinuum}{\mathfrak{c}}

\newcommand{\yosetf}{\mathcal{F}}

\newcommand{\yogeomqorigin}{\zeta}
\newcommand{\yogeomq}[1]{\yogeomqorigin_{(#1)}}
\newcommand{\yogeomqt}[2]{\yogeomqorigin_{(#1), #2}}

\newcommand{\yonfunc}{F}

\newcommand{\yoqqq}{Q}

\newcommand{\yobbb}[1]{H_{#1}}
\newcommand{\yoiii}[1]{I_{#1}}

\newcommand{\yominq}[1]{\mu(#1)}

\newcommand{\yoprism}[1]{\Psi_{#1}}

\newcommand{\yoxset}{\mathbb{X}}
\newcommand{\yokset}{K}

\newcommand{\yominus}{\ominus}

\newcommand{\yoipt}{\mathrm{LI}}

\newcommand{\yoyizz}{N(\yoyi)}

\newcommand{\yosc}{\mathrm{Suc}(\yocontinuum)}

\newcommand{\yoccset}{\mathbb{G}}

\newcommand{\yomark}{(\mathrm{M})}

\newcommand{\yotil}[1]{\widetilde{#1}}
\newcommand{\yoproj}[2]{\pi_{#1}(#2)}

\newcommand{\yoppset}{\mathrm{P}}

\newcommand{\yoinduce}[3]{[#1]_{#2, #3}}
\newcommand{\yoinducet}[2]{[#1]_{#2}}

\newcommand{\yollbra}{[\![}
\newcommand{\yorrbra}{]\!]}
\newcommand{\yoinducett}[2]{\yollbra#1\yorrbra_{#2}}

\newcommand{\yodim}{\mathrm{dim}}
\newcommand{\yolid}{\mathrm{Ind}}
\newcommand{\yosid}{\mathrm{ind}}

\makeatletter
\@addtoreset{equation}{section}

\makeatother

\begin{document}

\vspace{0.5in}

\renewcommand{\bf}{\bfseries}
\renewcommand{\sc}{\scshape}
\vspace{0.5in}

\title[Strongly rigid metrics]%
{Strongly
rigid metrics in 
spaces of metrics}

\author{Yoshito Ishiki}
\address{\endgraf
Photonics Control Technology Team
\endgraf
RIKEN Center for Advanced Photonics
\endgraf
2-1 Hirasawa, Wako, Saitama 351-0198, Japan}
\email{yoshito.ishiki@riken.jp}


%


\begin{abstract}
A metric space is said to be 
strongly rigid if no positive distance 
is taken twice by the metric. 
In 1972,  
Janos 
proved that 
a separable metrizable space has a 
strongly rigid metric if and only if 
it is zero-dimensional. 
In this paper, 
we shall develop this result for 
the theory of  spaces of metrics. 
For a strongly zero-dimensional 
metrizable space, 
we prove  that 
the set of all strongly rigid 
metrics is dense in the space of metics. 
Moreover, if the space is the union of 
countably many  compact subspaces, 
then that set is comeager. 
As a consequence, 
we show that 
for a strongly zero-dimensional 
metrizable space, 
the set of all metrics
possessing no nontrivial (bijective) self-isometry 
is 
comeager in the space of 
metrics. 
\end{abstract}

\maketitle

\section{\bf Introduction}

\subsection{Background}
Let $X$ be a topological space, 
and  $S$  a subset of $[0, \infty)$
with $0\in S$. 
We denote by $\met(X; S)$
the set of all metrics on $X$ taking values in $S$ and  generating  the same topology 
of $X$. 
We also denote by $\metdis_{X}$ the supremum metric on $\met(X; S)$; 
 namely, 
$\metdis_{X}(d, e)=\sup_{x, y\in X}|d(x, y)-e(x, y)|$.
We often write $\met(X)=\met(X; [0, \infty))$. 
Observe  that 
$\metdis_{X}$ is a metric taking values in 
$[0, \infty]$. 
As is the case of ordinary metric spaces, 
we can introduce the topology on $\met(X)$  
generated  
by open balls. 
In what follows, we consider that 
$\met(X)$ is equipped with this topology. 
In \cite{Ishiki2020int, Ishiki2021ultra, ishiki2021dense, Ishiki2022comeager}, 
the author 
proved the denseness and determined 
the Borel hierarchy of 
a subset of $\met(X)$ which can be represented as 
$\{\, d\in \met(X)\mid \text{$d$ satisfies $P$}\, \}$
for some property $P$ on metric spaces under 
certain conditions. 
For example, in \cite{ishiki2021dense}, 
the author proved that 
the set of all doubling metrics in $\met(X)$ is 
dense and $F_{\sigma}$ for 
every compact
finite-dimensional metrizable space 
$X$.

A metric $d$ on a set  $X$ is 
said to be 
\emph{strongly rigid} if 
for all $x, y, u, v\in X$, 
the relations 
$d(x, y)=d(u, v)$ and $d(x, y)\neq 0$ imply 
$\{x, y\}=\{u,  v\}$.

In 1972, 
Janos \cite{MR288739} proved that
a separable metric space $X$ is 
strongly $0$-dimensional if and only if 
there exists a strongly rigid metric $d\in \met(X)$
(see also  
\cite{MR454938}). 
The existence of strongly rigid metrics 
affected  
 research on a characterization of  the dimension of 
metrizable spaces using values of metrics 
(see, for example, \cite{MR314012},  \cite{MR474229} and \cite{MR1000155}).

A topological space $X$ is said to be 
\emph{strongly $0$-dimensional} if 
for every pair $A, B$ of disjoint closed subsets of $X$, 
there exists a clopen subset $V$ of $X$ such that 
$A\subset V$ and $V\cap B=\emptyset$. 
Such a space is sometimes said 
 to be 
\emph{ultranormal}.

In this paper, 
we develop 
the result on the existence of strongly rigid metrics 
for  the theory of 
spaces of metrics. 
For a strongly 
$0$-dimensional 
metrizable space, 
we prove  that 
the set of all strongly rigid 
metrics is dense in the space of metics. 
Moreover, if the space is $\sigma$-compact,  
then that set is $G_{\delta}$. 
As a consequence, 
we show that 
for a strongly 
$0$-dimensional 
metrizable space, 
the set of all metrics
possessing no nontrivial (bijective) self-isometry 
is 
comeager in the space of 
metrics.

\subsection{Main results}

The symbol ``$\yocontinuum$'' stands for 
the cardinality of the continuum. 
For a set $S$, 
we denote by $\card(S)$ the cardinality of 
$S$. 
A subset of a topological space is said to be 
$G_{\delta}$ if it is the intersection of 
countably many  open subsets.

Let $X$ be a metrizable space. 
We denote by
$\yoipt(X)$
the set of all metrics $d$ such that 
if  $x, y, u, v\in X$ satisfies 
$x\neq y$, $u\neq v$, and 
$\{x, y\}\neq \{u, v\}$, 
then $d(x, y)$ and $d(u, v)$
are linearly independent 
over $\qq$. 
The following is our 
first result:

\begin{thm}\label{thm:main1}
Let $X$ be a strongly $0$-dimensional 
metrizable space
with $\card(X)\le \yocontinuum$. 
Let $\epsilon \in (0, \infty)$ 
and $d\in \met(X)$. 
Then there exists 
$e\in \yoipt(S)$
such that 
$\metdis_{X}(d, e)\le \epsilon$. 
Namely, 
the set $\yoipt(X)$ is 
dense in 
$(\met(X), \metdis_{X})$. 
Moreover, 
if $X$ is completely metrizable, 
we can choose $e$ as a complete metric. 
\end{thm}

We denote by 
$\yosr(X)$
the set of all strongly rigid metrics in  $\met(X)$. 
As a consequence of Theorem \ref{thm:main1}, 
we obtain our second result:
\begin{thm}\label{thm:main15}
Let $X$ be a strongly $0$-dimensional 
metrizable space
with $\card(X)\le \yocontinuum$. 
Then the set $\yosr(X)$ is 
dense in $\met(X)$. 
Moreover, 
if $X$ is $\sigma$-compact, 
then $\yosr(X)$ is 
dense $G_{\delta}$ in $(\met(X), \metdis_{X})$.
\end{thm}

\begin{rmk}
Note that 
Theorem \ref{thm:main15} is true 
even if $X$ is not locally compact. 
For example, Theorem \ref{thm:main15} is 
ture for $X=\qq$. 
\end{rmk}

\begin{rmk}
Theorem \ref{thm:main15} 
can be considered as 
an analogue of Rouyer's result that 
generic metric spaces in the Gromov--Hausdorff space are strongly rigid
\cite[Theorem 2]{MR2831899}. 
Rouyer uses the term ``totally anisometric'' instead of 
``strongly rigid''. 
\end{rmk}

We say that a metric $d$ on a set $X$ is
said to be \emph{rigid} if 
every bijective isometry 
$f\colon (X, d)\to (X, d)$ must be 
the identity map. 
We denote by $\yorrr(X)$
the set of all rigid metrics in $\met(X)$.

Let $X$ be a topological space. 
A subset $S$ of $X$  is said  to be 
\emph{comeager} if $S$ contains 
a dense $G_{\delta}$ subset of $X$. 
As an application of Theorem \ref{thm:main1}, 
we obtain our third result: 
\begin{thm}\label{thm:main2}
Let $X$ be a strongly $0$-dimensional 
metrizable space. 
If $X$ is $\sigma$-compact and satisfies 
$3\le \card(X)\le \yocontinuum$, 
then 
the set $\yorrr(X)$ is comeager in 
$(\met(X), \metdis_{X})$. 
\end{thm}
\begin{rmk}
Our Theorem \ref{thm:main2} can be considered as 
a $0$-dimensional 
 analogue of the result that 
the set of all  rigid Riemannian (or pseudo-Riemannian)
metrics 
 is 
open dense in the space of Riemannian metrics with respect to the Whitney $C^{\infty}$-topology
(see \cite{MR0267604} and \cite{MR3320922}). 
\end{rmk}

As a consequence of Theorem \ref{thm:main15}, 
we know the existence of strange  metrics on 
second-countable  locally compact Hausdorff spaces. 
\begin{thm}\label{thm:plus1}
Let $X$ be a
strongly $0$-dimensional  
second-countable  locally compact Hausdorff space. 
Then there exists $d\in \met(X)$ such that 
for every $\xi\in X$, 
the map $F_{\xi}\colon X\to [0, \infty)$ defined by 
$F_{\xi}(x)=d(x, \xi)$ is a topological embedding. 
\end{thm}

\begin{rmk}
If we omit the assumption that 
$X$ is strongly $0$-dimensional,
then our main results  may not hold. 
For example, 
focusing on 
Theorem 
\ref{thm:main15}, 
we obtain
$\yosr(\rr)=\emptyset$, 
where, 
of course, 
the space $\rr$ is equipped with 
the Euclidean topology. 
In other words,  for every $d\in \met(\rr)$, 
there exist $x, y, u, v \in \rr$ with 
$x\neq y$ and $u\neq v$
 such that 
$d(x, y)=d(u, v)$
and $\{x, y\}\neq \{u, v\}$. 
This is a corollary of 
 the intermediate value theorem. 
 Furthermore, 
 it is known that 
if a metrizable space $X$ satisfies $\yosr(X)\neq \emptyset$, 
 then $X$ has 
  small inductive dimension $0$, 
 which means that $X$ has an open base
consisting of 
clopen subsets
 (see \cite[Lemma 2.1]{MR288739}). 
\end{rmk}

\begin{rmk}
There are three classical topological dimensions, 
the small inductive dimension $\yosid(X)$, 
the large inductive dimension 
$\yolid(X)$, 
and 
the covering dimension $\yodim(X)$
(see \cite{MR0394604}).
Note that  every metrizable  space $X$ satisfies 
$\yolid(X)=\yodim(X)$
and 
 a metrizable  space  $X$ is strongly $0$-dimensional if and only if 
$\yolid(X)=\yodim(X)=0$. 
Our definition of the strong $0$-dimensionality is 
nothing but $\yolid(X)=0$. 
For every metrizable space $X$, 
we have $\yosid(X)\le \yolid(X)$
and if $X$ is separable, 
then $\yosid(X)=\yolid(X)$. 
However, 
in general, for a non-separable 
metrizable space $X$, 
the small inductive dimension
$\yosid(X)$ is not equal to 
$\yolid(X)$. 
In particular, 
the condition $\yosid(X)=0$ 
does not  necessarily imply 
$\yolid(X)=0$. 
There is   a famous 
(completely)
metrizable space $\Delta$
constructed  in  
\cite{MR142102} and \cite{MR227960} 
(see also \cite{MR0394604})
called Roy's example, 
which 
satisfies 
$\yosid(\Delta)=0$ and 
$\yolid(\Delta)=\yodim(\Delta)=1$. 
With regard  to the 
difference between $\yosid(X)$ and 
$\yolid(X)$, 
in \cite{MR474229}, 
Janos and Matin said that 
it appeared to be  difficult to show that 
every metrizable space $X$ with 
 $\yosid(X)=0$ and $\card(X)\le \yocontinuum$
 admits 
 a strongly rigid metric, i.e., $\yosr(X)\neq \emptyset$.
 This question seems to be  still open as of now. 
Note that 
this statement is true if we replace the assumption $\yosid(X)=0$  with 
$\yolid(X)=0$ 
(see \cite{MR454938} and Corollary \ref{cor:cardsr} in the present  paper)
and 
the condition $\yosr(X)\neq\emptyset$ implies 
$\yosid(X)=0$
 (see \cite[Lemma 2.1]{MR288739}). 
The author does not know whether Roy's metrizable space $\Delta$
(or any space $X$ with $\yosid(X)=0$)
admits a strongly rigid metic, 
nor 
does the author 
 know 
whether we can replace the assumption that $X$ is strongly 
$0$-dimensional with the condition that 
$\yosid(X)=0$ in the statements of our main results. 
\end{rmk}

The organization of this paper is as follows: 
In Section 
\ref{sec:discrete}, 
we introduce a ubiquitously dense subset of 
a metrizable space, and 
we construct a strongly rigid discrete metric 
taking values in a given ubiquitously dense subset 
of $[0, \infty)$. 
In Section \ref{sec:qq}, 
we give a system of 
linearly independent 
real numbers over $\qq$
using a bijection between 
$\zz_{\ge 0}$ and $\qq_{\ge 0}$. 
The argument in this section 
is 
devoted  to showing 
Lemma \ref{lem:xyzabc}, 
which is a important part of the proof of 
Theorem \ref{thm:main1}. 
In Section \ref{sec:rigid}, 
we first construct a strongly rigid metric on 
the countable power $\yoyizz$ of the 
 discrete space $\yoyi$ with 
  $\card(\yoyi)=\yocontinuum$. 
Since 
 every strongly $0$-dimensional 
metrizable space $X$ with $\card(X)\le \yocontinuum$ can be topologically embedded
into $\yoyizz$, 
we can obtain a
strongly rigid metric on $X$. 
In Section 
\ref{sec:proof}, 
we prove Theorems \ref{thm:main1}, 
\ref{thm:main15}, \ref{thm:main2}, 
and \ref{thm:plus1}.

\section{\bf A construction of strongly rigid discrete  metrics}\label{sec:discrete}
In this paper, 
we say that a metric is \emph{discrete}
if it generates the discrete topology. 
We first give  a 
construction of strongly rigid 
metrics on discrete spaces. 
We begin with the following basic 
proposition on the triangle inequality. 
\begin{prop}\label{prop:triangular}
Let $N_{1}, N_{2}, N_{3}\in \zz_{\ge 1}$. 
We assume that  $M_{1}, M_{2}, M_{3}\in (0, \infty)$ satisfy 
$M_{i}\in (N_{i}+2^{-N_{i}-1}, N_{i}+2^{-N_{i}})$ for all 
 $i\in \{1, 2, 3\}$. 
If $N_{1}\le N_{2}+N_{3}$, 
then we have $M_{1}< M_{2}+M_{3}$. 
\end{prop}
\begin{proof}
We may assume that $N_{2}\le N_{3}$. 
If $N_{1}<N_{2}+N_{3}$, 
then we have 
$M_{1}< N_{1}+1\le N_{2}+N_{3}<
M_{2}+M_{3}$. 

In the case of $N_{1}=N_{2}+N_{3}$, 
 we have 
\begin{align*}
M_{1}&<N_{1}+2^{-N_{1}}
=N_{1}+2^{-(N_{2}+N_{3})}
\le N_{1}+2^{-N_{3}}\\
&=
N_{1}+2^{-N_{3}-1}+2^{-N_{3}-1}
\le 
 N_{2}+N_{3}+2^{-N_{2}-1}+2^{-N_{3}-1}
\\
&=(N_{2}+2^{-N_{2}-1})+(N_{3}+2^{-N_{3}-1})
<M_{2}+M_{3}. 
\end{align*}
Thus, 
we conclude that 
$M_{1}<M_{2}+M_{3}$. 
\end{proof}
\begin{rmk}
A crucial point of Proposition \ref{prop:triangular} is 
that we can choose $M_{i}$ depending  only on $N_{i}$. 
\end{rmk}

\begin{df}\label{df:densedf}
Let $X$ be a metrizable space. 
We say that  a subset $S$ of $X$ is 
\emph{ubiquitously dense} if 
for every non-empty open subset $U$ of $X$, 
we have 
$\card(U\cap S)=\card(S)$. 
\end{df}
Note that if there exists a 
ubiquitously dense subset $S$ of a metrizable space $X$
with $1<\card(X)$, 
then $S$ should be infinite and  the space $X$ has no isolated points.

In this paper, 
we use the set-theoretic representation of 
cardinal. 
For example, 
the relation $\alpha<\yocontinuum$ means 
$\alpha\neq \yocontinuum$ and 
$\alpha\in \yocontinuum$, 
and we have $\yocontinuum=\{\, \alpha\mid \alpha<\yocontinuum\, \}$.
For more discussion, 
we refer the readers to 
\cite{MR1940513}.

As a related work to the next theorem, 
we refer the readers to \cite{MR4022760}
concerning a partition of  $\rr$ into an 
uncountable family of dense subsets. 
\begin{thm}\label{thm:densedecomposition}
Let $X$ be a separable metrizable space
with $1<\card(X)$
and 
 $S$  a ubiquitously dense subset of 
$X$. 
Put $\kappa =\card(S)$. 
Then there exists a 
family  $\{A(\alpha)\}_{\alpha<\kappa}$ of 
subsets of $X$
such that 
\begin{enumerate}
\item each $A(\alpha)$ is countable;  
\item each $A(\alpha)$ is dense in $X$; 
\item if $\alpha, \beta\in \kappa$ satisfy 
 $\alpha\neq\beta$, 
then $A(\alpha)\cap A(\beta)=\emptyset$; 
\item $S=\bigcup_{\alpha<\kappa}A(\alpha)$. 
\end{enumerate}
\end{thm}
\begin{proof}
Notice that $\aleph_{0}\le \kappa$. 
By $\kappa\times \aleph_{0}=\kappa$, 
there exists a (strictly)
linear order $\prec$ such that 
$(\kappa\times \zz_{\ge 0}, \prec)$ and 
$(\kappa, <)$ are isomorphic to each other 
as ordered sets. 
Since $X$ is separable, 
there exists a countable  open base 
$\{U_{i}\}_{i\in \zz_{\ge 0}}$ of $X$. 
Using  transfinite recursion, 
we shall
define a sequence 
$\{q(\alpha, i)\}_{(\alpha, i)\in \kappa\times \zz_{\ge 0}}$
such that 
\begin{enumerate}[label=\textup{(\Alph*)}]
\item\label{item:inda} if $(\alpha, i), (\beta, j)\in \kappa\times \zz_{\ge 0}$ satisfy $(\alpha, i)\neq (\beta, j)$, 
then $q(\alpha, i)\neq q(\beta, j)$;
\item\label{item:indb} for all $(\alpha, i)\in \kappa\times \zz_{\ge 0}$, 
we have $q(\alpha, i)\in U_{i}\cap S$. 
\end{enumerate}
We assume that 
we have  already obtained a 
 sequence 
$\{q(\alpha, i)\}_{(\alpha, i)\prec (\theta, n)}$
satisfying the conditions
\ref{item:inda} and \ref{item:indb} for the set
$\{\, (\alpha, i)\mid (\alpha, i)\prec (\theta, n)\, \}$
instead of $\kappa\times \zz_{\ge 0}$. 
We shall construct $q(\theta, n) \in X$. 
Put $B=\{\, q(\alpha, i)\mid (\alpha, i)\prec (\theta, n) \, \}$. 
Since $(\kappa\times \zz_{\ge 0}, \prec)$
is isomorphic to $(\kappa, <)$, 
we have 
$\card(B)<\kappa$. 
Thus, from $\card(U_{n}\cap S)=\kappa$, 
it follows that
$(U_{n}\cap S)\setminus B\neq \emptyset$. 
Take $q(\theta, n)\in (U_{n}\cap S)\setminus B$. 
Then $\{q(\theta, n)\}\cup B$ satisfies 
the 
conditions (A) and (B)
 for the set
$\{(\theta, n)\}\cup \{\, (\alpha, i)\mid (\alpha, i)\prec (\theta, n)\, \}$
instead of $\kappa\times \zz_{\ge 0}$.
Therefore, by transfinite 
recursion, 
we obtain a sequence 
$\{q(\alpha, i)\}_{(\alpha, i)\in \kappa\times \zz_{\ge 0}}$
satisfying the condition (A) and (B).

For each $\alpha\in \kappa$, 
we define 
$B(\alpha)=\{\, q(\alpha, i)\mid i\in \zz_{\ge 0}\, \}$. 
Then the conditions (A) and (B) imply that 
the family  $\{B(\alpha)\}_{\alpha<\yocontinuum}$ satisfies the following conditions:
\begin{enumerate}[label=(\alph*)]
\item each $B(\alpha)$ is a subset of $S$;
\item if $\alpha, \beta\in \yocontinuum$ satisfy $\alpha\neq \beta$, then 
$B(\alpha)\cap B(\beta)=\emptyset$;
\item each $B(\alpha)$ is countable and dense in 
$X$. 
\end{enumerate}

We put
$\tau=
\card(S\setminus \bigcup_{\alpha<\kappa}B(\alpha))$ 
and 
we
take a bijection 
$\xi\colon \tau\to S\setminus \bigcup_{\alpha<\kappa}B_{\alpha}$. 
Notice that $\tau\le \kappa$. 
For each $\alpha<\kappa$, 
we define $A(\alpha)$ by 
$A(\alpha)=B(\alpha)\cup\{\xi(\alpha)\}$ if 
$\alpha<\tau$; 
otherwise,  $A(\alpha)=B(\alpha)$. 
Then the family $\{A(\alpha)\}_{\alpha<\kappa}$
is as desired. 
\end{proof}

The following
lemma is identical with \cite[Proposition 2.5]{Ishiki2022comeager}, 
which is related to 
metric-preserving functions.

\begin{lem}\label{lem:ceiling}
Let $X$ be a discrete topological space. 
Let $\eta\in (0, \infty)$
and 
$d\in \met(X)$.
 Then, there exists 
a metric $e\in \met(X; \eta\cdot \zz)$ 
such that 
 $\metdis_{X}(d, e)\le \eta$
 and $\eta\le e(x, y)$
 for all distinct 
 $x, y\in X$. 

\end{lem}

For a set $X$, 
we define 
$[X]^{2}$ by 
$[X]^{2}
=\{\, \{x, y\}\mid x, y\in X, x\neq y\, \}$. 
Note that if $X$ is infinite, 
we have $\card(X)=\card([X]^{2})$. 
A metric $d$ on a set $X$ is said to be 
\emph{uniformly discrete} if 
there exists $c\in (0, \infty)$ such that 
$c<d(x, y)$ for all distinct $x, y\in X$.
\begin{thm}\label{thm:approxcontinuum}
Let $S$ be a ubiquitously  dense subset of $[0, \infty)$
with $0\in S$ and 
$\kappa=\card(S)$. 
Let $X$ be a discrete space 
with $\card(X)\le \kappa$
and $d\in \met(X)$.
If $\epsilon \in (0, \infty)$, 
then there exists a
metric $e\in\met(X;  S)$
such that 
\begin{enumerate}[label=\textup{(\arabic*)}]
\item\label{item:24:1} we have $\metdis_{X}(d, e)\le \epsilon$; 
\item\label{item:24:2}  the metric $e$ is uniformly discrete;
\item\label{item:24:3}  we have 
$e(x, y)<e(x, z)+e(z, y)$ for all distinct $x, y, z\in X$; 
\item\label{item24:4} the metric $e$ is strongly rigid. 
\end{enumerate}
\end{thm}
\begin{proof}
We put $\eta=\epsilon/2$
and $T=\eta^{-1}\cdot S=\{\, \eta^{-1} s\mid s\in S\, \}$. 
Note that $T$ is ubiquitously dense 
in $[0, \infty)$ and $0\in T$. 

Lemma \ref{lem:ceiling} guarantees the existence 
of 
 a metric 
$h\in \met(X; \eta\cdot \zz_{\ge 0})$ 
such that $\metdis_{X}(h, e)\le \eta$. 
Put $u=\eta^{-1}\cdot h\in \met(X; \zz_{\ge 0})$. 
Due to Theorem \ref{thm:densedecomposition}, 
we can take a mutually disjoint dense decomposition 
$\{A(\alpha)\}_{\alpha<\kappa}$
 of 
$T$. 
Put $\tau=\card(X)$. 
Then $\tau\le \kappa$. 
We take a bijection 
$\varphi\colon \tau\to [X]^{2}$. 
We represent 
$\varphi(\alpha)=(x_{\alpha}, y_{\alpha})$
and $\theta_{\{x, y\}}=\varphi^{-1}(\{x, y\})$. 
For each $\alpha<\tau$, 
we put 
$N_{\alpha}=u(x_{\alpha}, y_{\alpha})\in \zz_{\ge 1}$
and 
take 
 $w(\alpha)\in (N_{\alpha}+2^{-N_{\alpha}-1}, N_{\alpha}+2^{-N_{\alpha}})\cap A(\alpha)$. 
 Since each $A(\alpha)$ is dense in $[0, \infty)$, 
the existence of   $w(\alpha)$ is 
always guaranteed.

 We define a function $v\colon X^{2}\to [0, \infty)$
  by $v(x, y)=w(\theta_{\{x, y\}})$ 
  if $x\neq y$; 
  otherwise, $v(x, x)=0$. 
According to  Proposition \ref{prop:triangular}, 
the function  
$v\colon X^{2}\to [0, \infty)$ satisfies the triangle inequality. 
Since $1+2^{-2}\le v(x, y)$ for all distinct $x, y\in X$, 
the metric $v$ generates the discrete topology on $X$, 
namely, 
$v\in \met(X; T)$. 
Notice that $\metdis_{X}(u, v)\le 1$. 

Put $e=\eta\cdot v $. 
Since $v$ is in $\met(X)$, 
so is $e$. 
By $v\in \met(X; T)$ and $T=\eta^{-1}\cdot S$, 
we have $e\in \met(X; S)$. 

Using 
$h=\eta\cdot u$, $e=\eta\cdot v$, 
$\metdis_{X}(d, h)\le \eta$, and 
$\metdis_{X}(u, v)\le 1$, 
we obtain 
\begin{align*}
\metdis_{X}(d, e)&\le \metdis_{X}(d, h)+\metdis_{X}(h, e)=
\metdis_{X}(d, h)+\metdis_{X}(\eta\cdot u, \eta\cdot v)
\\
&\le
\eta+\eta\metdis_{X}(u, v)\le 
2\eta=\epsilon.
\end{align*}
This implies  condition \ref{item:24:1}. 

Since $1+2^{-2}\le v(x, y)$ for all distinct $x, y\in X$, we observe that $e(=\eta\cdot v)$ is uniformly discrete. 
This means that   condition \ref{item:24:2} is true.

From  Proposition \ref{prop:triangular}, 
it follows  that 
$e(x, y)<e(x, z)+e(z, y)$ for all distinct 
$x, y, z\in X$.
This proves   condition
 \ref{item:24:3}. 

We shall prove $e$ is strongly rigid. 
Take 
$\{x, y\}, \{a, b\}\in [X]^{2}$ with 
$\{x, y\}\neq \{a, b\}$. 
Then 
$\theta_{\{x, y\}}\neq \theta_{\{a, b\}}$
and 
$A\left(\theta_{\{x, y\}}\right)
\cap 
A\left(\theta_{\{a, b\}}\right)=\emptyset$. 
In particular, we have 
$w\left(\theta_{\{x, y\}}\right)\neq 
w\left(\theta_{\{a, b\}}\right)$, 
and hence $e(x, y)\neq e(a, b)$. Namely, 
the metric 
$e$ is strongly rigid. 
This implies that $e$ satisfies  condition 
\ref{item24:4}. 
\end{proof}

\begin{rmk}
In the case where  $X$ is finite, 
Theorem \ref{thm:approxcontinuum} gives 
a new proof of \cite[Lemma 3]{MR2831899} stating that the set of finite metric spaces which are
totally anisometric (strongly rigid) and without collinear points is dense in the Gromov--Hausdorff space. 
\end{rmk}

\section{\bf A system of linearly independent numbers over $\qq$}\label{sec:qq}
In this subsection, we 
give a system yielding 
 real numbers which are 
 linearly independent   over $\qq$. 

\begin{df}\label{df:inprod}
Let $\alpha=\{a_{i}\}_{i\in \zz_{\ge 0}}$
 be a summable sequence of positive real numbers. 
Let $\yoqqq\colon \zz_{\ge 0}\to \qq_{\ge 0}$ be a
bijection. 
For a non-empty subset $B$ of $\qq_{\ge 0}$, 
we define 
\[
\yoinpq{\alpha}{B}=\sum_{\yoqqq(i)\in B}a_{i}. 
\]
If $B=\emptyset$, 
we also define $\yoinpq{\alpha}{\emptyset}=0$. 
\end{df}

\begin{df}\label{df:fastmap}
We denote by 
$\yofastset$
the set of all 
 $\yonfunc\colon \zz_{\ge 0}\to \zz_{\ge 0}$ 
such that $F$ is strictly increasing and
satisfies 
\[
\lim_{n\to \infty}(\yonfunc(n+1)-\yonfunc(n))=\infty,  
\]
and 
\[
\lim_{n\to \infty}\sum_{m=n+1}^{\infty}2^{\yonfunc(n)-\yonfunc(m)}=0. 
\]
\end{df}
The author is inspired by 
\cite{mathof}
(see also \cite{MR1512442} and \cite{MR173645}) with respect to a  construction of  linearly independent real numbers over $\qq$ in the following proposition:
\begin{prop}\label{prop:independent5}
Let $\yonfunc\in \yofastset$. 
We define a sequence $\lambda=\{\lambda_{i}\}_{i\in \zz_{\ge }}$ by $\lambda_{i}=2^{-\yonfunc(i)}$. 
Let $k\in \zz_{\ge 0}$ and
 $\{P_{i}\}_{i=0}^{k}$  a 
family of subsets of $\qq_{\ge 0}$. 
Let $S$ be a subset of $\qq_{\ge 0}$. 
We assume that 
 there exist $a$, $b_{0}, \dots, b_{k}$ in $[0, \infty)$ such that 
\begin{enumerate}[label=\textup{(\Roman*)}]
\item $a<b_{i}$ for all $i\in \{0, \dots, k\}$; 
\item if $i\neq j$, then  $b_{i}\neq b_{j}$;
\item 
we have 
$S\cap P_{i}=[a, b_{i})\cap \qq_{\ge 0}$ for all 
$i\in \{0, \dots, k\}$. 
\end{enumerate}
Then the $(k+2)$-many numbers 
$\yoinpq{\lambda}{P_{0}}, \dots, \yoinpq{\lambda}{P_{k}}$,  and $1$ are linearly independent 
over $\qq$. 
\end{prop}
\begin{proof}
We may 
 assume that 
 $b_{i}<b_{i+1}$ for all 
 $i\in \{0, \dots, k-1\}$. 
For every $i\in \{0, \dots, k\}$, 
we define 
$r(i)=\yoinpq{\lambda}{P_{i}}$. 
To prove the proposition, 
we assume that  integers 
$c_{0}, \dots, c_{k}, c_{k+1}\in \zz$ 
satisfy 
\begin{align}\label{al:indep}
c_{0}\cdot r(0)+\dots +c_{k}\cdot r(k)+c_{k+1}\cdot 1=0. 
\end{align}
We first prove $c_{k}=0$. 
Since $S\cap P_{k}=[a, b_{k})\cap \qq$ and 
$b_{k-1}<b_{k}$, we observe that 
the set $P_{k}\setminus \left(\bigcup_{i=0}^{k-1}P_{i}\right)$ is infinite. 
Thus, 
by $\yonfunc\in \yofastset$, 
we can 
take a sufficient large $n\in \zz_{\ge 0}$ such that 
\begin{enumerate}[label=\textup{(\arabic*)}]
\item\label{item:<1} $(|c_{0}|+\dots +|c_{k}|)\sum_{m=n+1}^{\infty}2^{\yonfunc(n)-\yonfunc(m)}<1$. 
\item\label{item:pm}
$n\in P_{k}$
\item\label{item:p-m} 
we have 
$n\not\in \bigcup_{i=0}^{k-1}P_{i}$. 
\item\label{item:small}
$|c_{k}|<2^{\yonfunc(n)-\yonfunc(n-1)}$. 
\end{enumerate}
Put 
\[
I=c_{0}\sum_{j\in P_{i}\cap [0, n]}2^{\yonfunc(n)-\yonfunc(j)}
+\dots+
c_{k}\sum_{j\in P_{i}\cap [0, n]}
2^{\yonfunc(n)-\yonfunc(j)}
+c_{k+1}2^{\yonfunc(n)}. 
\]
and 
\[
J
=c_{0}\sum_{j\in P_{i}\cap (n, \infty)}
2^{\yonfunc(n)-\yonfunc(j)}
+\dots+
c_{k}\sum_{j\in P_{i}\cap (n, \infty)}2^{\yonfunc(n)-\yonfunc(j)}. 
\]
Then the equality \eqref{al:indep} implies 
\[
I+J=2^{\yonfunc(n)}\times
(c_{0}r(0)+\dots +c_{k}r(k)+c_{k+1})=0. 
\]

By \ref{item:<1}, we have $|J|<1$. 
Note that $I\in \zz$
(if $j\le n$, then $\yonfunc(n)-\yonfunc(j)\ge 0$). 
 By $J=-I$, we obtain 
 $J\in \zz$. 
 Combining $J\in \zz$ and $|J|<1$, 
we conclude that $J=0$. 
Thus, we
also  have $I=0$. 

Since for every $j\le n-1$
we have $\yonfunc(n)-\yonfunc(n-1)\le \yonfunc(n)-\yonfunc(j)$,
the number 
$2^{\yonfunc(n)-\yonfunc(j)}$ can be divided by 
$2^{\yonfunc(n)-\yonfunc(n-1)}$. 
Thus, by $n\in P_{k}$, 
 we have 
\begin{align}\label{al:1+m}
\sum_{j\in P_{k}\cap [0, n]}2^{\yonfunc(n)-\yonfunc(j)}
=1+L_{k}\cdot 2^{\yonfunc(n)-\yonfunc(n-1)}
\end{align}
for some $L_{k}\in \zz$. 
Due to \ref{item:p-m}, 
 for all $i$ with $i<k$, 
we have  $n\not\in P_{i}$. 
Hence 
\begin{align}\label{al:0+m}
\sum_{j\in P_{i}\cap [0, n]}
2^{\yonfunc(n)-\yonfunc(j)}
=L_{i}\cdot 2^{\yonfunc(n)-\yonfunc(n-1)}
\end{align}
for some $L_{i}\in \zz$. 
Since $c_{k+1}$ is an integer, 
and since $\yonfunc(n)-\yonfunc(n-1)<\yonfunc(n)$, 
\begin{align}\label{al:somek+1}
c_{k+1}2^{\yonfunc(n)}=L_{k+1}\cdot 2^{\yonfunc(n)-\yonfunc(n-1)}
\end{align}
for some $L_{k+1}\in \zz$ 
(of cause, 
$L_{k+1}=c_{k+1}2^{\yonfunc(n-1)}$). 

From  $I=0$, 
and 
the equalities 
\eqref{al:1+m}, 
\eqref{al:0+m},
and \eqref{al:somek+1}, 
it follows that 
\[
c_{k}=M\cdot 2^{\yonfunc(n)-\yonfunc(n-1)}
\]
for some $M\in \zz$. 
By $|c_{k}|<2^{\yonfunc(n)-\yonfunc(n-1)}$
(the condition \ref{item:small}), 
we conclude that  $M=0$. 
Thus $c_{k}=0$. 
Using the same argument, by induction, 
 we obtain 
$c_{k}=c_{k-1}=\dots=c_{0}=0$. 
Hence $c_{k+1}=0$. 
This means that 
the numbers 
$\yoinpq{\lambda}{P_{0}}, \dots, \yoinpq{\lambda}{P_{k}}$ and $1$ are linearly independent 
over $\qq$. 
\end{proof}

\section{\bf A construction of strongly rigid metrics}\label{sec:rigid}
In this subsection, 
we construct a strongly rigid metric on 
a given  strongly $0$-dimensional 
metrizable 
space.

 The following 
 proposition 
 is deduced from 
 the existence and uniqueness of 
 the binary representation of a real number
 with infinitely  many digits which are $1$. 
 We omit the proof. 

\begin{prop}\label{prop:base222}
Let $f\colon \zz_{\ge 0}\to \zz_{\ge 0}$ be 
an injective map
and
define  a sequence $\lambda=\{\lambda_{i}\}_{i\in \zz_{\ge 0}}$  
 by 
$\lambda_{i}=2^{-f(i)}$. 
Fix a bijection $\yoqqq\colon \zz_{\ge 0}\to \qq_{\ge 0}$. 
Let $S$ and $T$ be 
infinite or empty subsets of 
$\qq_{\ge 0}$. 
If $\yoinpq{\lambda}{S}=
\yoinpq{\lambda}{T}$, 
then we have 
$S=T$. 
\end{prop}

 \begin{df}\label{df:Fk}
Let $k\in \zz_{0}$. 
 We define 
 $\yonfunc_{k}\colon \zz_{\ge 0}\to \zz_{\ge 0}$ by 
 $\yonfunc_{k}(n)=
 2^{n}+k$. 
Observe  that $\yonfunc_{k}\in \yofastset$. 
 We also 
 define a sequence 
 $\yogeomq{k}=\{\yogeomqt{k}{i}\}_{i\in \zz_{\ge 0}}$ 
 by 
 $\yogeomqt{k}{i}=2^{-\yonfunc_{k}(i)}$. 
 \end{df}

\begin{df}\label{df:yoyi}
In what follows, 
we fix the discrete space $\yoyi$ 
such that 
$\card(\yoyi)=\yocontinuum$. 
We define 
$\yoyizz=\yoyi^{\zz_{\ge 0}}$. 
We consider that 
the set $\yoyizz$ is always  equipped with 
the product topology. 
\end{df}
Remark  that the space $\yoyizz$ is sometimes called 
the  \emph{Baire space of weight $\yocontinuum$}
(see \cite{MR0394604}
or 
\cite{MR152457}). 

Our first purpose  is 
to construct 
strongly rigid metrics on 
$\yoyi$ and  $\yoyizz$. 
For this purpose, 
we utilize 
strongly rigid semi-metrics.

\begin{df}\label{df:semimetric}
Let $S$ be a subset of $[0, \infty)$. 
We say that a map 
$r\colon X\times X\to [0, \infty)$ is 
an 
\emph{$S$-semi-metric on $X$} if 
\begin{enumerate}[label=\textup{(\arabic*)}]
\item for all $x, y\in X$, we have $r(x, y)=r(y, x)$;
\item for all $x, y\in X$, we have 
$r(x, y)=0$ if and only if $x=y$; 
\item for all distinct $x, y\in X$, we have $r(x, y)\in S$. 
\end{enumerate}
The \emph{strong rigidity of an  $S$-semi-metric} is 
defined  just as with ordinary metrics. 
\end{df}
Note that semi-metrics are not assumed  to satisfy 
the triangle inequality.

\begin{prop}\label{prop:existssrsm}
Let $S$ be a subset $(0, \infty)$ with $\card(S)=\yocontinuum$. 
Then there exists a strongly rigid 
$S$-semi-metric $r$ on $\yoyi$.
\end{prop}
\begin{proof}
Since $\card([\yoyi]^{2})=\yocontinuum$, 
we can take a bijection $\phi\colon [\yoyi]^{2}\to S$. 
We define 
$r(x, y)=\phi(\{x, y\})$ if $x\neq y$, 
otherwise, $r(x, x)=0$. 
Then $r$ is as desired. 
\end{proof}

\begin{df}\label{df:propertyq}
Let $\yoqqq\colon \zz_{\ge 0}\to \qq_{\ge 0}$
be a bijection. 
We define 
$\yominq{m}=\min\yoqqq^{-1}([m, m+1)\cap \qq)$. 
We say that the map $\yoqqq$ satisfies the
\emph{property $\yomark$} if 
\begin{enumerate}
\item for all $m\in \zz_{\ge 0}$, 
we have $Q(\yominq{m})=m$; 
\item for all $m\in \zz_{\ge 0}$, we have 
$\yominq{m}<\yominq{m+1}$. 
\end{enumerate}
Notice that $\yominq{0}=0$. 
\end{df}

\begin{lem}\label{lem:existsheart}
There exists a bijection 
$\yoqqq\colon \zz_{\ge 0}\to \qq_{\ge 0}$ 
satisfying the property 
$\yomark$. 
\end{lem}
\begin{proof}
Take a mutually disjoint family 
$\{A_{i}\}_{i\in \zz_{\ge 0}}$ of 
infinite subsets of $\zz_{\ge 0}$
satisfying that 
$\bigcup_{i\in \zz_{\ge 0}}A_{i}=\zz_{\ge 0}$
and 
$\min A_{i}<\min A_{i+1}$ for all $i\in \zz_{\ge 0}$. 
For each $i\in \zz_{\ge 0}$, 
we take a bijection $\theta_{i}\colon A_{i}\to [i, i+1)\cap \qq$ with $\theta_{i}(\min A_{i})=i$. 
Gluing  them together, we obtain a 
bijection $\zz_{\ge 0}\to \qq_{\ge 0}$ with 
the property $\yomark$. 
\end{proof}

\begin{rmk}
In what follows, 
based on 
Lemma \ref{lem:existsheart}, 
 we fix the bijection $\yoqqq\colon \zz_{\ge 0}\to \qq_{\ge 0}$ with the property 
 $\yomark$. 
\end{rmk}

\begin{df}\label{df:Jset}
Fix $k\in \zz_{\ge 0}$. 
Let $m\in \zz_{\ge 0}$
and $r$ an $S$-semi-metric on 
$\yoyi$. 
We define 
$\yobbb{m, r(x, y)}=[m, r(x, y))\cap \qq$. 
Note that if $x=y$, 
the set 
$\yobbb{m, r(x, y)}$ is always 
empty. 
We also define $\yoinduce{r}{k}{m}\colon \yoyi\times \yoyi\to 
[0, \infty)$ by 
$\yoinduce{r}{k}{m}(x, y)=
\yoinpq{\yogeomq{k}}{\yobbb{m, r(x, y)}}$. 
\end{df}
\begin{rmk}
Under the same assumptions as 
in Definition \ref{df:Jset}, 
we notice that 
$\yobbb{0, r(x, x)}=[0, 0)\cap \qq=\emptyset$
for all $x\in \yoyi$. 
\end{rmk}

 \begin{lem}\label{lem:stst}
 Fix $k\in \zz_{\ge 0}$. 
  Let $m\in \zz_{\ge 0}$
  and  
$S$ a subset of $(m, m+1)$. 
  Let $X$ be a discrete space
 with $\card(X)\le \yocontinuum$. 
If $r$ is a strongly rigid 
 $S$-semi-metric on $\yoyi$, 
then the following statements are true:
\begin{enumerate}[label=\textup{(\arabic*)}]
\item\label{item:210:1}
We have 
$\yoinduce{r}{k}{m}(x, y)\in (\yogeomqt{k}{\yominq{m}}, 
2\yogeomqt{k}{\yominq{m}})$ for all 
distinct $x, y\in X$. 
\item\label{item:210:2}
The function $\yoinduce{r}{k}{m}$ is a metric on $\yoyi$ and 
we have 
 $\yoinduce{r}{k}{m}(x, y)\in \met(\yoyi)$. 
\item\label{item:210:3}
The metric $\yoinduce{r}{k}{m}$ is strongly rigid. 
\end{enumerate}
 \end{lem}
 \begin{proof}
 We first prove the statement \ref{item:210:1}. 
 Take distinct $x, y\in \yoyi$. 
By the property $\yomark$, 
we have 
 $\yoqqq(\yominq{m})=m$ and 
 $m\in [m, d(x, y))\cap \qq=\yobbb{m, r(x, y)}$. 
 Thus we obtain
\[
\yogeomqt{k}{\yominq{m}}\le \yoinpq{\yogeomq{k}}{\yobbb{m, r(x, y)}}=\yoinduce{r}{k}{m}(x, y). 
\] 
Since $\yominq{m}$ is the minimal number of 
the set 
 $\yoqqq^{-1}([m, m+1)\cap \qq)$
 (see Definition \ref{df:propertyq}),
 and since $\yonfunc_{k}$ is strictly 
 increasing, 
 we have 
 \begin{align*}
 \yoinpq{\yogeomq{k}}
 {\yobbb{m, r(x, y)}}
&=
 \sum_{i\in \yoqqq^{-1}([m, m+1)\cap \qq)}2^{-\yonfunc_{k}(i)}
 \le 
 \sum_{\yonfunc_{k}(\yominq{m})\le i}2^{-i}
\\
& = 
 2\cdot 2^{-\yonfunc_{k}(\yominq{m})}
 =2\yogeomqt{k}{\yominq{m}}. 
 \end{align*}
Hence
 $\yoinpq{\yogeomq{k}}
 {\yobbb{m, r(x, y)}}\le 
 2\yogeomqt{k}{\yominq{m}}$. 
 This implies the statement 
 \ref{item:210:1}.
 
 We next prove the statement \ref{item:210:2}. 
 If $x, y\in \yoyi$
 satisfies  $x=y$, we have $\yobbb{m, r(x, y)}=\emptyset$. 
 Thus $\yoinduce{r}{k}{m}(x, y)=
 \yoinpq{\yogeomq{k}}{\yobbb{m, r(x, y)}}=0$.
To show that 
 $\yoinduce{r}{k}{m}$ satisfies the triangle inequality, 
 we take  distinct $x, y, z\in \yoyi$. 
According to  the statement \ref{item:210:1}, 
 we have 
 \[
 \yoinduce{r}{k}{m}(x, y)\le 2\yogeomqt{k}{\yominq{m}}
 =\yogeomqt{k}{\yominq{m}}+\yogeomqt{k}{\yominq{m}}
 <\yoinduce{r}{k}{m}(x, z)+\yoinduce{r}{k}{m}(z, y). 
 \]
 Thus, the function $\yoinduce{r}{k}{m}$ satisfies the 
 triangle inequality. 
According to  the statement \ref{item:210:1} again, 
the metric  $\yoinduce{r}{k}{m}$ is uniformly discrete, 
and hence  $\yoinduce{r}{k}{m}\in \met(\yoyi)$. 
 This finishes the proof of \ref{item:210:2}. 
 
 We shall prove the statement \ref{item:210:3}. 
 We assume that $x, y, u, v\in X$ satisfy
 $0<\yoinduce{r}{k}{m}(x, y)$ and 
 $\yoinduce{r}{k}{m}(x, y)=\yoinduce{r}{k}{m}(u, v)$. 
 Then we have 
 $\yoinpq{\yogeomq{k}}{\yobbb{m, r(x, y)}}
=\yoinpq{\yogeomq{k}}{\yobbb{m, r(u, v)}}$. 
By Proposition \ref{prop:base222}, 
we obtain 
$\yobbb{m, r(x, y)}=\yobbb{m, r(u, v)}$. 
Namely, we have 
$[m, r(x, y))\cap \qq=[m, r(u, v))\cap \qq$, 
and hence
$r(x, y)=r(u, v)$. 
Since $r$ is strongly rigid, 
we conclude that  $\{x, y\}=\{u, v\}$. 
This means that $d$ is strongly rigid. 
Then the statement  \ref{item:210:3} is true. 
This completes the proof of the lemma. 
 \end{proof}

\begin{df}\label{df:system}
Let $S$ be a dense subset of 
$[0, \infty)$. 
We say that 
a 
family $\{r_{i}\}_{i\in \zz_{\ge 0}}$  is 
an
\emph{$S$-gauge system on $\yoyi$}
if 
 each $r_{i}$ is a strongly rigid 
$(S\cap (i, i+1))$-semi-metric on $\yoyi$.
\end{df}

\begin{prop}\label{prop:existssys}
Let $S$ be a 
ubiquitously dense subset in $[0, \infty)$ with 
$\card(S)=\yocontinuum$. 
Then, there exists an 
$S$-gauge system 
$\{r_{i}\}_{i\in \zz_{\ge 0}}$
on $\yoyi$. 
\end{prop}
\begin{proof}
Since $S$ is ubiquitously dense, 
we have $\card(S\cap (i, i+1))=\yocontinuum$
for all $i\in \zz_{\ge 0}$. 
Then, 
by Proposition \ref{prop:existssrsm}, 
there exists a 
strongly rigid $(S\cap (i, i+1))$-semi-metric
$r_{i}$
on $\yoyi$. 
Thus, the sequence 
$\{r_{i}\}_{i\in\zz_{\ge 0}}$ is 
an $S$-gauge system
on $\yoyi$. 
\end{proof}

\begin{df}\label{df:Iset}
Fix $k\in \zz_{\ge 0}$. 
Let $S$ be a ubiquitously dense subset of $[0, \infty)$ and 
 $R=\{r_{i}\}_{i\in\zz_{\ge 0}}$ 
 an 
$S$-gauge system on $\yoyi$. 
In this case, 
$\yobbb{m, r_{m}(x_{m}, y_{m})}\cap \yobbb{m^{\prime}, r_{m^{\prime}}(x_{m^{\prime}}, y_{m^{\prime}})}=\emptyset$ for all 
distinct $m, m^{\prime}\in \zz_{\ge 0}$ and for all
$x=(x_{i})_{i\in \zz_{\ge 0}}$ and $y=(y_{i})_{i\in \zz_{\ge 0}}$ 
in 
$\yoyizz$. 
We define 
\[
\yoiii{R, x, y}=
\coprod_{m\in \zz_{\ge 0}}\yobbb{m, r_{m}(x_{m}, y_{m})}. 
\]
We also define a function 
$\yoinducet{R}{k}\colon
 \yoyizz\times \yoyizz
 \to [0, \infty)$ by 
\[
\yoinducet{R}{k}(x, y)=\yoinpq{\yogeomq{k}}{\yoiii{R, x, y}}. 
\]
\end{df}

Under the same assumptions
as in 
Definition \ref{df:Iset}, 
notice that we have 
\[
\yoinducet{R}{k}(x, y)=
\sum_{i=0}^{\infty}\yoinduce{r_{m}}{k}{m}(x_{m}, y_{m}), 
\]
where $x=(x_{i})_{i\in \zz_{\ge 0}}$ and 
$y=(y_{i})_{i\in \zz_{\ge 0}}$. 
Since $\yogeomq{k}$ is 
summable, we have 
$\yoinducet{R}{k}(x, y)<\infty$ for all 
$x, y\in \yoyizz$. 

\begin{lem}\label{lem:abab}
 Fix $k\in \zz_{\ge 0}$. 
 Let $S$ be a ubiquitously dense subset of $[0, \infty)$
 and 
 $R=\{r_{i}\}_{i\in \zz_{\ge 0}}$ 
 an
 $S$-gauge system on $\yoyi$. 
 Let $m\in \zz_{\ge 0}$
 and 
 $x=(x_{i})_{i\in \zz_{\ge 0}}, 
 y=(y_{i})_{i\in \zz_{\ge 0}}\in 
 \yoyizz$. 
Then the following statements 
hold:
\begin{enumerate}[label=\textup{(\Alph*)}]
\item\label{item:proda}
If $x, y\in X$ satisfy $\yoinducet{R}{k}(x, y)\le \yogeomqt{k}{\yominq{m}}$, 
then we have $x_{i}=y_{i}$ for all $i \in \{0, \dots, m\}$. 
\item\label{item:prodb}
If $x_{i}=y_{i}$ for all $i \in \{0, \dots, m\}$, 
then we have $\yoinducet{R}{k}(x, y)\le 4\yogeomqt{k}{\yominq{m+1}}$. 
\end{enumerate}

\end{lem}
\begin{proof}
We first prove \ref{item:proda}. 
For the sake of contradiction, 
we suppose that  there exists 
$i\in \{0, \dots, m\}$
such that 
$x_{i}\neq y_{i}$ and 
$\yoinducet{R}{k}(x, y)\le \yogeomqt{k}{\yominq{m}}$. 
Since $\yoqqq$ satisfies the 
property $\yomark$, 
we have $\yominq{i}\le \yominq{m}$, and hence 
$\yogeomqt{k}{\yominq{m}}\le \yogeomqt{k}{\yominq{i}}$. 
By $\yoinducet{R}{k}(x, y)=
\sum_{i=0}^{\infty}\yoinduce{r_{m}}{k}{m}(x_{m}, y_{m})$
and by $\yogeomqt{k}{\yominq{i}}< \yoinduce{r_{i}}{k}{i}(x_{i}, y_{i})$ (see the statement  \ref{item:210:1} in Lemma \ref{lem:stst}), 
we have $\yogeomqt{k}{\yominq{m}}<\yoinducet{R}{k}(x, y)$. 
This is a contradiction. 
Therefore the statement \ref{item:proda} is true. 

We next prove \ref{item:prodb}. 
If $x_{i}=y_{i}$ for all $i \in \{0, \dots, m\}$, 
we have $\yoinducet{R}{k}(x, y)=
\sum_{i=m+1}^{\infty}\yoinduce{r_{m}}{k}{m}(x_{m}, y_{m})$. 
According to the statement  \ref{item:210:1} in Lemma \ref{lem:stst}, 
for each $i\in\zz_{\ge m+1}$, 
we have 
$\yoinduce{r_{m}}{k}{m}(x_{m}, y_{m})<
2\yogeomqt{k}{\yominq{i}}$. 
Since  we have 
$\yoinducet{R}{k}(x, y)=
\sum_{i=0}^{\infty}\yoinduce{r_{m}}{k}{m}(x_{m}, y_{m})$, 
and 
since 
$\sum_{i=m+1}^{\infty}\yogeomqt{k}{\yominq{i}}\le 
2\yogeomqt{k}{\yominq{m+1}}$, 
we obtain 
$\yoinducet{R}{k}(x, y)
\le 4\yogeomqt{k}{\yominq{m+1}}$. 
This proves  the statement \ref{item:prodb}. 
\end{proof}

\begin{prop}\label{prop:compe}
 Fix $k\in \zz_{\ge 0}$. 
 Let $S$ be a ubiquitously dense subset of $[0, \infty)$
 and $R=\{r_{i}\}_{i\in \zz_{\ge 0}}$ 
 an
 $S$-gauge system on $\yoyi$. 
 Then the function 
 $\yoinducet{R}{k}\colon 
 \yoyizz\times \yoyizz 
 \to [0, \infty)$ satisfies the 
 following statements:
\begin{enumerate}[label=\textup{(\arabic*)}]
\item\label{item:212:1}
We have 
$\yoinducet{R}{k}\in \met\left(\yoyizz\right)$. 
\item\label{item:212:2}
We have 
$\yodiam_{\yoinducet{R}{k}}\left(\yoyizz\right)
\le 2\cdot 
2^{-\yonfunc_{k}(0)}(=2^{-k})$. 
\item\label{item:212:3}
The metric 
$\yoinducet{R}{k}$ is complete.  
\end{enumerate}
\end{prop}
\begin{proof}
We first prove the statement \ref{item:212:1}. 
Since each $\yoinduce{r_{i}}{k}{i}$ generates the same topology on $\yoyi$, the 
statements \ref{item:proda} and 
\ref{item:prodb} in Lemma \ref{lem:abab}
imply that 
$\yoinducet{R}{k}$ generates the same topology on 
$\yoyizz$ (recall that $\yoyizz$ is equipped with the 
product topology). 

We next prove the statement \ref{item:212:2}. 
For all $x, y\in \yoyizz$, we have 
\[
\yoinducet{R}{k}(x, y)\le \yoinpq{\yogeomq{k}}{\qq_{\ge 0}}
=\sum_{i\in \zz_{\ge0 }}2^{-\yonfunc_{k}(i)}
\le \sum_{\yonfunc_{k}(0)\le j}2^{-j}=2\cdot 2^{-\yonfunc_{k}(0)}. 
\]
Since  $\yonfunc_{k}(0)=1+k$, 
and $\yoinducet{R}{k}(x, y)\le 2\cdot 2^{-\yonfunc_{k}(0)}$, 
 the statement \ref{item:212:2} is true. 

To prove the statement \ref{item:212:3}, 
we take an arbitrary  Cauchy sequence 
$\{a(n)\}_{n\in \zz_{\ge 0}}$ in 
$\yoyizz$. 
We put 
$a(n)=(a(n, i))_{i\in \zz_{\ge 0}}$
for all $i\in \zz_{\ge 0}$ and $n\in \zz_{\ge 0}$. 
Due to the statement  \ref{item:proda} in 
Lemma \ref{lem:abab}, 
for each $i\in \zz_{\ge 0}$, 
 there exists $N_{i}$ such that 
$a(n, i)=a(n+1, i)$  for all $n\in \zz_{\ge 0}$ with 
$N_{i}\le n$. 
We define   $b_{i}=a(N_{i}, i)$. 
Then the point 
$b=(b_{i})_{i\in \zz_{\ge 0}}\in \yoyizz$ is a
limit of $\{a(n)\}_{n\in \zz_{\ge 0}}$. 
Thus, the metric $\yoinducet{R}{k}$ is complete. 
This finishes the proof of the proposition. 
\end{proof}
\begin{rmk}
The metric 
$\yoinducet{R}{k}$ in 
Definition \ref{df:Iset} is not strongly rigid. 
Indeed, if we take $a, b\in \yoyi$ with $a\neq b$, 
and  define 
$x, y, u, v \in \yoyizz$ by 
$x=(a, a, a, \dots, )$, 
$y=(a, b, b, \dots, )$, 
$u=(b, a, a, \dots, )$, and 
$v=(b, b, b, \dots, )$, 
then 
we have $\{x, y\}\neq \{u, v\}$, and 
$\yoinducet{R}{k}(x, y)=\yoinducet{R}{k}(u, v)$
for a fixed integer $k\in \zz_{\ge 0}$, and for every $S$-gauge system $R$ on $\yoyi$. 
\end{rmk}

To obtain a strongly rigid metric 
on $\yoyizz$, 
we construct a topological embedding 
$\yoprism{\yosetf}\colon
 \yoyizz
\to \yoyizz$
such that 
$\yoprism{\yosetf}(x)$ contains 
information about  all finite prefixes of 
$x=(x_{i})_{i\in \zz_{\ge 0}}$ in $\yoyizz$. 

\begin{df}\label{df:prism}
For $n\in \zz_{\ge 0}$, and for a point 
$x=(x_{i})_{i\in \zz_{\ge 0}}\in \yoyizz$, 
we denote by 
$\yoproj{n}{x}\in \yoyi^{n+1}$ the point 
$(x_{0}, x_{1}, \dots, x_{n})$. 
For each $i\in \zz_{\ge 0}$, 
we take an injective map 
$f_{i}\colon \yoyi^{i+1}\to \yoyi$. 
Since 
$\card(\yoyi)=\yocontinuum$ and $\yocontinuum^{i+1}=\yocontinuum$, 
the injective map  $f_{i}\colon \yoyi^{i+1}\to \yoyi$ always exists. 
Put
$\yosetf=\{f_{i}\}_{i\in\zz_{\ge 0}}$. 
We define a
 map
$\yoprism{\yosetf}\colon
 \yoyizz
\to \yoyizz$ as follows: 
The $i$-th entry $y_{i}$ of $\yoprism{\yosetf}(x)$ is 
defined by 
\[
y_{i}=
\begin{cases}
x_{n} & \text{if $i=2n$};\\ 
f_{n}(\yoproj{n}{x})& \text{if $i=2n+1$, } 
\end{cases}
\]
where $x=(x_{i})_{i\in \zz_{\ge 0}}$. 
In what follows, 
we fix the family $\yosetf=\{f_{i}\}_{i\in \zz_{\ge 0}}$. 
\end{df}

\begin{lem}\label{lem:closedtopemb}
The map 
$\yoprism{\yosetf}\colon \yoyizz\to 
\yoyizz$ is 
a topological embedding and 
the image of $\yoprism{\yosetf}$ is 
closed in $\yoyizz$. 
\end{lem}
\begin{proof}
By the definition of $\yoprism{\yosetf}$, 
we observe that $\yoprism{\yosetf}$ is 
a topological embedding. 
To prove that 
$\yoprism{\yosetf}(\yoyizz)$ is closed, 
we take a sequence $\{x_{i}\}_{i\in \zz_{\ge 0}}$ 
in $\yoprism{\yosetf}(\yoyizz)$ such that 
$x_{i}\to a$ as $i\to \infty$ for some point 
$a=(a_{i})_{i\in \zz_{\ge 0}}\in \yoyizz$. 
We define $b=(b_{i})_{i\in \zz_{\ge 0}}\in \yoyizz$
by $b_{i}=a_{2i}$. 
Then $\yoprism{\yosetf}(b)=a$ and hence 
$a\in \yoprism{\yosetf}(\yoyizz)$. 
Thus, the set 
$\yoprism{\yosetf}(\yoyizz)$ is closed in $\yoyizz$. 
\end{proof}

Since for all 
$x=(x_{i})_{i\in \zz_{\ge 0}}, y=(y_{i})_{i\in \zz_{\ge 0}}\in \yoyizz$
we have 
$x=y$ if and only if 
$x_{i}=y_{i}$
for all 
$i\in \zz_{\ge 0}$, 
we obtain the following lemma:
\begin{lem}\label{lem:coherent}
Let 
$l\in \zz_{\ge 0}$
and 
$p(0), \dots p(l)\in \yoyizz$. 
If 
$p(0), \dots p(l)$ 
are mutually distinct, 
then 
there exists 
$N\in \zz_{\ge 0}$ such that 
for all 
$n\in \zz_{\ge 0}$ with $N<n$, the points 
$\yoproj{n}{p(0)}, \dots 
\yoproj{n}{p(l)}$ 
are mutually 
distinct. 
\end{lem}
As a consequence of 
Lemma \ref{lem:coherent}, 
we obtain: 
\begin{lem}\label{lem:coherent2}
Let 
$l\in \zz_{\ge 0}$
and 
$x(0), y(0), x(1), y(1), \dots, x(l), y(l)\in \yoyizz$. 
If 
$x(s)\neq y(s)$ 
for all 
$s\in \{0, \dots, l\}$, 
and 
if 
$\{x(s), y(s)\}\neq \{x(t), y(t)\}$ 
for all distinct 
$s,  t\in \{0, \dots, l\}$, 
then there exists $N$ 
such that for all 
$n\in \zz_{\ge 0}$ with $N<n$, 
\begin{enumerate}
\item 
we have 
$\{\yoproj{n}{x(s)}, \yoproj{n}{y(s)}\}\neq 
\{\yoproj{n}{x(t)}, \yoproj{n}{y(t)}\}$
for all distinct $s, t\in \{0, \dots, l\}$. 
\item we have $\yoproj{n}{x(s)}\neq 
\yoproj{n}{y(s)}$ for all $s\in \{0, \dots, l\}$. 
\end{enumerate}
\end{lem}
\begin{proof}
Put $x(s)=(x_{i}(s))_{i\in \zz_{\ge 0}}$ and 
$y(s)=(y_{i}(s))_{i\in \zz_{\ge 0}}$. 
For each $s \in \{0, \dots, l\}$,
we define $u(s)=(u_{i}(s))_{i\in \zz_{\ge 0}}$ 
and $v(s)=(v_{i}(s))_{i\in \zz_{\ge 0}}$ by 
\[
u_{i}(s)
=
\begin{cases}
x_{j}(s) & \text{if $i=2j$;}\\
y_{j}(s) & \text{if $i=2j+1$,}
\end{cases}
\]
and 
\[
v_{i}(s)
=
\begin{cases}
y_{j}(s) & \text{if $i=2j$;}\\
x_{j}(s) & \text{if $i=2j+1$.}
\end{cases}
\]
The inequality $x(s)\neq y(s)$ implies 
$u(s)\neq v(s)$. 
For all distinct $s, t\in \{0, \dots, l\}$,
by the assumption  that $\{x(s), y(s)\}\neq \{x(t), y(t)\}$, 
the points  $u(s), v(s), u(t), v(t)$
are mutually distinct. 
Therefore the points 
$u(0)$, $v(0)$, \dots, $u(l), v(l)$
are mutually distinct. 
According to  Lemma \ref{lem:coherent2}, 
there exists $M\in \zz_{\ge 0}$ such that 
for all $m\in \zz_{\ge0}$ with $M<m$, 
the points 
$\yoproj{m}{u(0)}, \yoproj{m}{v(0)}, 
\dots, \yoproj{m}{u(l)}, \yoproj{m}{v(l)}$
are mutually distinct. 
Take $N\in \zz_{\ge 0}$ with 
$M\le 2N$. 
Then, 
by the definitions of 
$u(s)$ and $v(s)$, 
the integer  $N$ satisfies the 
two conditions stated in the proposition. 
\end{proof}

Using the metric  $\yoinducet{R}{k}$ and 
the embedding $\yoprism{\yosetf}$, 
we construct a strongly rigid metric on 
$\yoyizz$. 
\begin{df}\label{df:inducemetric}
 Let $S$ be a subset of $[0, \infty)$
 and let 
$R=\{r_{i}\}_{i\in \zz_{\ge 0}}$ be an $S$-gauge system on 
$\yoyi$. 
We define 
\[
\yoinducett{R}{k}(x, y)=
\yoinducet{R}{k}(\yoprism{\yosetf}(x), \yoprism{\yosetf}(y)). 
\]
\end{df}
Note that we have 
\[
\yoinducett{R}{k}(x, y)=
\yoinpq{\yogeomq{k}}{
\yoiii{R, \yoprism{\yosetf}(x), \yoprism{\yosetf}(y)}}. 
\]

For a metric space $(X, d)$, 
and for a subset $A$ of $X$, 
we denote by 
$\yodiam_{d}(A)$ the diameter of $A$ with 
respect to $d$. 
The following theorem plays an 
important role to prove 
our main results. 

\begin{thm}\label{thm:5555}
 Fix $k\in \zz_{\ge 0}$. 
  Let $S$ be a 
   ubiquitously dense subset of $[0, \infty)$, 
  and 
 $R=\{r_{i}\}_{i\in \zz_{\ge 0}}$ 
 an $S$-gauge system on 
$\yoyi$. 
 Then 
 the metric  
 $\yoinducett{R}{k}\colon \yoyizz\times \yoyizz\to [0, \infty)$ satisfies  the following statements:
\begin{enumerate}[label=\textup{(\arabic*)}]
\item\label{item:216:1}
We have 
$\yoinducett{R}{k}\in \met\left(\yoyizz\right)$. 

\item\label{item:216:2}
We have 
$\yodiam_{\yoinducett{R}{k}}\left(\yoyizz\right)\le 2\cdot 
2^{-\yonfunc_{k}(0)}(=2^{-k})$. 
\item\label{item:216:3}
The metric $\yoinducett{R}{k}$ is complete. 

\item\label{item:216:4}
The metric $\yoinducett{R}{k}$ is strongly rigid. 
\end{enumerate}

\end{thm}
\begin{proof}
Since $\yoprism{\yosetf}\colon \yoyizz\to 
\yoyizz$ is a 
topological embedding and its image is closed
(see Lemma \ref{lem:closedtopemb}), 
the statements \ref{item:212:1} and 
\ref{item:212:3} in Proposition \ref{prop:compe}
implies that  \ref{item:216:1} and 
\ref{item:216:3} in the theorem are true. 

By \ref{item:212:2}  in 
Proposition \ref{prop:compe}, 
the statement 
\ref{item:216:2} is true.

 We now prove the statement 
 \ref{item:216:4} in the theorem. 
Take $x, y, u, v\in \yoyizz$ such that
$x\neq y$, $u\neq v$, and 
$\{x, y\}\neq \{u, v\}$. 
Due to  Lemma \ref{lem:coherent2}, 
we can take $n\in \zz_{\ge 0}$ such that 
$\{\yoproj{n}{x}, \yoproj{n}{y}\}\neq 
\{\yoproj{u}{n}, \yoproj{n}{v}\}$.
Thus, 
using the injectivity of $f_{n}$, 
we have 
\[
\{f_{n}(\yoproj{n}{x}), f_{n}(\yoproj{n}{y})\}
\neq 
\{f_{n}(\yoproj{n}{u}), f_{n}(\yoproj{n}{v})\}.
\]
Then 
we have 
\[
r_{2n+1}(f_{n}(\yoproj{n}{x}), 
f_{n}(\yoproj{n}{y}))\neq 
r_{2n+1}(f_{n}(\yoproj{n}{u}), f_{n}(\yoproj{n}{v})).
\] 
This inequality implies that 
\[
\yobbb{2n+1, r_{2n+1}(f_{n}(\yoproj{n}{x}), f_{n}(\yoproj{n}{y}))}
\neq \yobbb{2n+1, r_{2n+1}(f_{n}(\yoproj{n}{u}), 
f_{n}(\yoproj{n}{v}))}.
\] 
Therefore,  we obtain 
\[
\yoiii{R, \yoprism{\yosetf}(x), \yoprism{\yosetf}(y)}\neq 
\yoiii{R, \yoprism{\yosetf}(x), \yoprism{\yosetf}(y)}. 
\]
According to  Proposition \ref{prop:base222}, 
we notice that
$\yoinducett{R}{k}(x, y)\neq \yoinducett{R}{k}(u, v)$. 
This means  that the metric 
$\yoinducett{R}{k}$ is strongly rigid, and 
 the statement \ref{item:216:4} is true. 
This completes the proof of the theorem. 
\end{proof}

\begin{df}
We denote by 
$\yosc$ the set $\yocontinuum\cup \{\yocontinuum\}$. 
This is nothing but  the successor of $\yocontinuum$ as an
ordinal. 
\end{df}

 \begin{lem}\label{lem:decodeco}
 There exists a family 
 $\{\yokset(\alpha)\}_{\alpha\in \yosc}$ 
 satisfying that 
 \begin{enumerate}[label=\textup{(\arabic*)}]
 \item 
 each $\yokset(\alpha)$ 
 is a subset of 
 $[0, \infty)$;
 \item 
each $\yokset(\alpha)$
  is ubiquitously dense $ in [0, \infty)$
  and $\card(\yokset(\alpha))=\yocontinuum$;
 \item 
 if $\alpha, \beta \in \yosc$
 satisfy $\alpha\neq \beta$, then 
 $\yokset(\alpha)\cap \yokset(\beta)=\emptyset$. 
 \end{enumerate}
 \end{lem}
 \begin{proof}
The set 
 $[0, \infty)$ 
 is ubiquitously dense 
 in $[0, \infty)$ and 
 $\card([0, \infty))=\yocontinuum$. 
 Thus, by
 Theorem \ref{thm:densedecomposition}, 
 we obtain a mutually disjoint decomposition 
 $\{T(\alpha)\}_{\alpha<\yocontinuum}$
 of $[0, \infty)$ such that 
 each $T(\alpha)$ is 
 countable and dense in $[0, \infty)$. 
 Since 
 $\card(\yosc)=\yocontinuum$, 
 and since $\yocontinuum\times \yocontinuum=\yocontinuum$ as 
  cardinals, 
 we can 
 take a bijection 
 $\phi\colon \yosc\to \yosc \times \yocontinuum$, 
 where 
 $\yosc \times \yocontinuum$
  stands for the product as sets.
 Put 
 $\phi(\alpha)=(\theta(\alpha), \lambda(\alpha))$. 
For each $\alpha\in \yosc$, 
 we define 
 $\yokset(\alpha)=\bigcup_{\beta\in \yosc,
  \theta(\beta)=\alpha}T(\beta)$. 
 Then the family 
 $\{\yokset(\alpha)\}_{\alpha\in \yosc}$
 satisfies the three conditions stated in the  lemma. 
 This finishes the proof. 
 \end{proof}

 \begin{df}\label{df:fandg}
 Fix $k\in \zz_{\ge 0}$. 
 In what follows,  in this paper, 
 we
 fix a family 
 $\{\yokset(\alpha)\}_{\alpha\in \yosc}$
 stated in Lemma \ref{lem:decodeco}. 
For each 
$\alpha\in \yosc$, 
we also fix 
a $\yokset(\alpha)$-gauge system 
$R(\alpha)=\{r_{i, \alpha}\}_{i\in \zz_{\ge 0}}$
(see Proposition \ref{prop:existssys}). 

 For each $\alpha\in \yosc$, 
we define 
 \[
 \yocset_{k, \alpha}=
 \{\, 
 \yoinducett{R(\alpha)}{k}(x, y)
 \mid x\neq y, x, y\in X\, \}, 
 \]
 and 
 \[
\yoccset_{k, \alpha}=\{0\}\sqcup \yocset_{k, \alpha}. 
 \]
Namely, the set 
$\yoccset_{k, \alpha}$
is the image of the metric 
$\yoinducett{R(\alpha)}{k}$ on $\yoyizz$. 
 \end{df}
 
In a similar way to the proof of  
\ref{item:216:4} in 
Theorem \ref{thm:5555}, 
we obtain 
Lemmas \ref{lem:eqeqeq} and 
\ref{lem:Findeqqq}. 
\begin{lem}\label{lem:eqeqeq}
 Fix $k\in \zz_{\ge 0}$. 
If $\alpha, \beta\in \yosc$
satisfy $\alpha\neq \beta$, 
then 
$\yocset_{k, \alpha}\cap \yocset_{k, \beta}=\emptyset$. 
 \end{lem}
\begin{proof}
For the sake of contradiction, 
we suppose that there exist $\alpha, \beta\in \yosc$
with $\alpha \neq \beta$ and $\yocset_{k, \alpha}\cap \yocset_{k, \beta}\neq\emptyset$. 
Take $d\in \yocset_{k, \alpha}\cap \yocset_{k, \beta}$
and put $d=\yoinducett{R(\alpha)}{k}(x, y)$ and 
$d=\yoinducett{R(\beta)}{k}(u, v)$, where $x, y, u, v\in \yoyizz$, 
and $x\neq y$ and $u\neq v$. 
Take $i\in \zz_{\ge 0}$ such that 
$x_{i}\neq y_{i}$. 
Then $r_{i, \alpha}(x_{i}, y_{i})\in K(\alpha)\setminus \{0\}$ and $r_{i, \beta}(u_{i}, v_{i})\in K(\beta)\cup \{0\}$. 
From $K(\alpha)\cap K(\beta)=\emptyset$, 
it follows that $r_{i, \alpha}(x_{i}, y_{i})\neq r_{i, \beta}(u_{i}, v_{i})$. 
This implies that 
$\yobbb{i, r_{i, \alpha}(x, y)}\neq \yobbb{i, r_{i, \beta}(u, v)}$ and hence 
$\yoiii{R(\alpha), \yoprism{\yosetf}(x), \yoprism{\yosetf}(y)}\neq \yoiii{R(\beta), \yoprism{\yosetf}(u), \yoprism{\yosetf}(v)}$. 
According to  Proposition \ref{prop:base222}, 
we obtain 
$\yoinducett{R(\alpha)}{k}(x, y)\neq \yoinducett{R(\beta)}{k}(u, v)$. 
This contradicts 
$d=\yoinducett{R(\alpha)}{k}(x, y)$ and 
$d=\yoinducett{R(\beta)}{k}(u, v)$. 
Therefore we conclude that 
$\yocset_{k, \alpha}\cap \yocset_{k, \beta}=\emptyset$
for all distinct $\alpha, \beta\in \yosc$. 
\end{proof}

A subset $S$ of $\rr$ is said to be 
\emph{linearly independent over $\qq$} if 
all finite subsets of  $S$ are linearly independent 
over $\qq$. 
 \begin{lem}\label{lem:Findeqqq}
 Fix $k\in \zz_{\ge 0}$. 
Let 
 $s, t\in \zz_{\ge 0}$. 
Let  
$\alpha_{0}, \dots, \alpha_{s}\in \yosc$. 
For each $i\in \{0, \dots, s\}$, 
we 
take $(t+1)$-many arbitrary distinct
elements  
$d_{i, 0}, \dots, d_{i, t}\in \yocset_{k, \alpha_{i}}$. 
Then the set 
\[
\{1\}\cup \{d_{i, j}\mid i\in \{0, \dots, s\}, 
j\in\{0, \dots, t\} \}
\]
is linearly independent over $\qq$. 
 \end{lem}
\begin{proof}

Put $d_{i, j}=\yoinducett{R(\alpha_{i})}{k}(x(i, j), y(i, j))$, 
where $x(i, j), y(i, j)\in \yoyizz$. 
Note that 
$x(i, j)\neq y(i, j)$. 
According to 
 Lemma \ref{lem:coherent2}, 
 we
can
take a sufficient large integer  $n\in \zz_{\ge 0}$ such that 
for each 
$i\in \{0, \dots, s\}$, 
for all distinct $j, j^{\prime}\in \{0, \dots, t\}$, 
we obtain the inequality 
\[
\{\yoproj{n}{x(i, j)}, \yoproj{n}{y(i, j)}\}
\neq \{\yoproj{n}{x(i, j^{\prime}}, 
\yoproj{n}{y(i, j^{\prime})}\}. 
\]
We put 
$r(i, j)=r_{2n+1, \alpha_{i}}(f_{n}(\yoproj{n}{x(i, j)}), 
f_{n}(\yoproj{n}{y(i, j)}))$. 
Then, 
for a  fixed number $i\in \{0, \dots, s\}$, for all distinct $j, j^{\prime}\in \{0, \dots, t\}$, 
we have $r(i, j)\neq r(i, j^{\prime})$. 
Since $r(i, j)\in \yokset(\alpha_{i})$ for all 
$j\in \{0, \dots, t\}$, 
and since $\yokset(\alpha_{i})\cap \yokset(\alpha_{i^{\prime}})=\emptyset$
for all distinct $i, i^{\prime}\in \{0, \dots, s\}$, 
we obtain $r(i, j)\neq r(i^{\prime}, j^{\prime})$
for all distinct $(i, j)$, $(i^{\prime}, j^{\prime})$. 

We put $I_{i, j}=\yoiii{R(\alpha_{i}), \yoprism{\yosetf}(x(i, j)), \yoprism{\yosetf}(y(i, j))}$, 
and 
$S=[2n+1, 2n+2)\cap \qq$. 
Then $S\cap I_{i, j}=[2n+1, r(i, j))\cap \qq$. 
Thus, 
by $d_{i, j}=\yoinpq{\yogeomq{k}}{I_{i, j}}$, 
the numbers 
$d_{i, j}$ $(i\in \{0, \dots, s\}, j\in\{0, \dots, t\})$
satisfy the 
assumptions of Proposition 
\ref{prop:independent5}. 
Therefore, 
due to Proposition 
\ref{prop:independent5}, 
we conclude that 
the set 
\[
\{1\}\cup \{d_{i, j}\mid i\in \{0, \dots, s\}, 
j\in\{0, \dots, t\} \}
\]
is linearly independent over $\qq$. 
\end{proof}

 \begin{cor}\label{cor:Findeq}
 Fix $k\in \zz_{\ge 0}$. 
Then the set 
$\{1\}\cup \bigcup_{\alpha\in \yosc}
\yocset_{k, \alpha}$
is linearly independent over $\qq$.  
 \end{cor}

 \begin{df}\label{df:Eset}
 Fix $k\in \zz_{\ge 0}$. 
 Recall that in Definition 
\ref{df:fandg}, 
we construct the family 
$\{\yocset_{k, \alpha}\}_{\alpha\in \yosc}$, 
where $\yosc=\yocontinuum\cup \{\yocontinuum\}$. 
 Using  $\yocontinuum\times \aleph_{0}=\yocontinuum$, 
 we can represent 
 \[
 \yocset_{k, \yocontinuum}=
 \left\{s_{k, \alpha, i}\mid 
 \alpha\in \yocontinuum, i\in \zz_{\ge 0}\right\}.
 \]
 We assume that  if 
 $s_{k, \alpha, i}=s_{k, \beta, j}$, 
then  $(\alpha, i)=(\beta, j)$. 
 For each 
 $(\alpha, i)\in \yocontinuum \times \zz_{\ge 0}$, 
 we take 
 $q_{k, \alpha, i}\in \qq_{>0}$
 such that 
 $q_{k, \alpha, i}\cdot s_{k, \alpha, i}\le 2^{-i}$. 
We fix  a bijection 
$P\colon \zz_{\ge 0}\to \qq_{\ge 0}$. 
We define set $\yobset_{k}$ and $\yoxset_{k}$ by 
\[
\yobset_{k}
=
\left\{\, 
P(i)+q_{k, \alpha, i}\cdot s_{k, \alpha, i}\mid \alpha<\yocontinuum, i\in \zz_{\ge 0}\, 
\right\},
\] 
and 
\[
\yoxset_{k}=\{0\}\sqcup \yobset_{k}, 
\]
respectively. 
In what follows, 
we no longer use the fact that 
the set $\yocset_{k, \yocontinuum}$ is
defined as a set of  values of a metric on $\yoyizz$. 
We rather use the property that 
the union of $\yocset_{k, \yocontinuum}$
and $\{1\}\cup \bigcup_{\alpha<\yocontinuum}\yocset_{k, \alpha}$ 
is linearly independent over $\qq$
(see Corollary \ref{cor:Findeq}). 
 \end{df}

 \begin{lem}\label{lem:eudense}
 Fix $k\in \zz_{\ge 0}$. 
 Then 
the set $\yoxset_{k}$ is 
 ubiquitously dense in 
 $[0, \infty)$ and 
 $\card(\yoxset_{k})=\yocontinuum$. 
 \end{lem}
 \begin{proof}
 It suffices to show that 
 $\yobset_{k}$ is 
 ubiquitously dense in 
 $[0, \infty)$. 
 Take $x\in [0, \infty)$
 and $\epsilon\in (0, \infty)$. 
 Since $\qq_{\ge 0}$ is dense in 
 $[0, \infty)$, 
 we can take $n\in \zz_{\ge 0}$ such that 
 $|P(n)-x|<\epsilon/2$ 
 and $2^{-n}\le \epsilon/2$. 
 Then, for all $\alpha<\yocontinuum$, 
 we have 
 \[
 \left|
 P(n)+q_{k, \alpha, n}\cdot s_{k, \alpha, n}-x
 \right|\le 
 |P(n)-x|+
 \left|q_{k, \alpha, n}\cdot s_{k, \alpha, n}\right|<\epsilon/2+2^{-n}\le \epsilon.
 \] 
 Thus, the set 
 $\yobset_{k}$ is ubiquitously dense
 in $[0, \infty)$. 
 \end{proof}
 
 By 
 Lemma \ref{lem:eqeqeq} and Corollary  \ref{cor:Findeq}, 
 and by 
 the definitions of 
 $\yobset_{k}$ and 
 $\yocset_{k, \alpha}$, 
 we obtain:
 \begin{prop}\label{prop:indkbeta}
 Fix $k\in \zz_{\ge 0}$. 
Then the following statements are true: 
\begin{enumerate}[label=\textup{(\arabic*)}]
\item\label{item:417/1}
For all $\alpha\in \yosc$, 
we have $\yobset_{k}\cap \yocset_{k, \alpha}=\emptyset$;
\item\label{item:417/2}
The set 
$\yobset_{k}\cup \bigcup_{\alpha<\yocontinuum}
\yocset_{k, \alpha}$
is linearly independent over $\qq$.  
\end{enumerate}
 \end{prop}

For two sets $A$ and $B$, 
we denote by 
$A\yominus B=(A\setminus B)\cup (B\setminus A)$. 
Namely, the set $A\yominus B$ is 
the symmetric difference of $A$ and $B$. 
\begin{lem}\label{lem:xyzabc}
Fix $k\in \zz_{\ge 0}$. 
Let  $\alpha, \beta, \yotil{\alpha}, 
\yotil{\beta} \in \yocontinuum$ with  
$\alpha\neq \beta$ and 
$\yotil{\alpha}\neq \yotil{\beta}$. 
If 
$x\in \yoccset_{k, \alpha}$, 
$a\in \yoccset_{k, \yotil{\alpha}}$
$y\in \yoccset_{k, \beta}$, 
$b\in \yoccset_{k, \yotil{\beta}}$
and 
$z, c\in \yoxset_{k}$. 
If 
$\{x, y, z\}\neq \{a, b, c\}$,  $x+y+z\neq 0$, 
and 
$a+b+c\neq 0$, 
then
the numbers 
$x+y+z$ and $a+b+c$
are linearly independent 
over $\qq$. 
\end{lem}
\begin{proof}
We first prove the following claim:
\begin{itemize}
\item There exists non-zero $r$ such that 
$r\in \{x, y, z\}\yominus\{a, b, c\}$.
\end{itemize}
By $\{x, y, z\}\neq \{a, b, c\}$, 
we observe that 
$\{x, y, z\}\yominus\{a, b, c\}\neq \emptyset$.
If $0\not \in \{x, y, z\}\yominus\{a, b, c\}$, 
then any element in this set satisfies the condition. 
If $0\in \{x, y, z\}\yominus\{a, b, c\}$, 
then $0\in \{x, y, z\}$ or $0\in \{a, b, c\}$.
We may assume that 
$0\in \{a, b, c\}$. 
Thus, $0\not\in \{x, y, z\}$. 
Combining 
Lemma \ref{lem:eqeqeq}
and the statement 
\ref{item:417/1} in Proposition \ref{prop:indkbeta}, 
we obtain 
$\yocset_{k, \alpha}\cap \yocset_{k, \beta}
=\yocset_{k, \alpha}\cap \yobset_{k}
=\yocset_{k, \beta}\cap \yobset_{k}=\emptyset$. 
Since $x\in \yocset_{k, \alpha}$, 
$y\in \yocset_{k, \beta}$, and 
$z\in \yobset_{k}$, 
we conclude that $\card(\{x, y, z\})=3$. 
Since $\{a, b, c\}$ contains at 
most two non-zero numbers, 
and since all the three numbers $x, y, z$ are non-zero, 
there exists non-zero 
$r\in \{x, y, z\}\yominus\{a, b, c\}$. 
This finishes the proof of the claim. 

To prove the linear independence
of $x+y+z$ and $a+b+c$ over $\qq$, 
we assume that integers 
$h_{0}$ and $h_{1}$ satisfy
\begin{align}\label{al:c0c1}
h_{0}(x+y+z)+h_{1}(a+b+c)=0. 
\end{align}
We
put $A=\{x, y, z\}\cap (0, \infty)$ and 
$B=\{a, b, c\}\cap (0, \infty)$.
We also put 
$u=\sum_{l\in A\setminus B}l$, 
$v=\sum_{l\in B\setminus A}l$, 
and $w=\sum_{l\in A\cap B}l$. 
Then the equality  \eqref{al:c0c1} implies that
\begin{align}\label{al:daidai}
h_{0}u+
h_{1}v
+(h_{0}+h_{1})w
=0. 
\end{align}
Observe 
 that it can happen that some of $u, v, w$ are $0$. 
Using  the claim explained above, 
we have $u\neq 0$ or $v\neq 0$. 
We may assume that $u\neq 0$. 
Put $C=\{u, v, w\}\cap (0, \infty)$. 
Then $u\in C$. 
According to  Proposition \ref{prop:indkbeta}, 
the set $A\cup B$ is linearly independent over $\qq$. 
Since $A\setminus B$, 
$B\setminus A$, and $A\cap B$
are mutually disjoint subsets of $A\cup B$, 
by the definitions of $u$, $v$,  and $w$, 
the set $C$ 
is linearly independent over $\qq$. 
Thus,
from \eqref{al:daidai} and 
and $u\neq 0$, 
it follows that  $h_{0}=0$. 
The equality \eqref{al:c0c1} implies 
 $h_{1}(a+b+c)=0$. 
Since $a+b+c\neq 0$, 
we obtain 
$h_{1}=0$. 
Thus, we conclude that  $h_{0}=h_{1}=0$. 
Hence $x+y+z$ and 
$a+b+c$ are linearly independent 
over $\qq$. 
\end{proof}

\begin{lem}\label{lem:approxEk}
Fix $k\in \zz_{\ge 0}$. 
Let $X$ be 
a discrete space with 
$\card(X)\le \yocontinuum$. 
Let $d\in \met(X)$
and $\epsilon\in (0, \infty)$. 
Then there exists a strongly rigid
uniformly discrete metric $e\in \met(X; \yoxset_{k})$ with 
$\metdis_{X}(d, e)\le \epsilon$. 
\end{lem}
\begin{proof}
Combining  Lemma \ref{lem:eudense} and 
Theorem \ref{thm:approxcontinuum}, 
we obtain  the lemma. 
\end{proof}

The following is deduced from 
\cite[Theorem 3.1, Chapter 7]{MR0394604}. 
The latter part is deduced from 
\cite[Theorem 2]{MR152457}.

\begin{thm}\label{thm:bc}
If  $X$ is  a strongly $0$-dimensional 
metrizable space 
with $\card(X)\le \yocontinuum$, 
then  $X$ can be topologically embedded into 
$\yoyizz$. 
Moreover, if $X$ is completely metrizable, 
the space 
$X$ is homeomorphic to a 
closed subset of $\yoyizz$. 
\end{thm}

\begin{lem}\label{lem:smallFk}
Fix $k\in \zz_{\ge 0}$. 
Let $X$ be 
a strongly $0$-dimensional 
metrizable space with 
$\card(X)\le \yocontinuum$. 
Let $\alpha\in \yocontinuum$. 
Then there exists a strongly rigid metric 
$e\in \met(X; \yoccset_{k, \alpha})$ 
such that  $\yodiam_{e}(X)\le 2^{-k}$. 
Moreover, if $X$ is completely metrizable, 
we can choose $e$ as a complete metric. 
\end{lem}
\begin{proof}
Due to Theorem \ref{thm:bc}, 
we can take a topological embedding 
$h\colon X\to \yoyizz$. 
Put $e(x, y)=\yoinducett{R(\alpha)}{k}(h(x), h(y))$. 
Then, 
 by 
Theorem \ref{thm:5555}, 
the metric $e$ satisfies the desired properties. 
If $X$ is completely metrizable, 
we can choose $h$ as a closed map. 
Since $\yoinducett{R}{k}$ is complete, 
in this case, 
so is the metric $e$.  
\end{proof}

The following was first stated in 
\cite{MR454938}. 
\begin{cor}\label{cor:cardsr}
Let $X$ be a strongly $0$-dimensional 
metrizable space. 
Then the inequality $\card(X)\le \yocontinuum$ holds 
if and only if 
there exists a strongly rigid metric 
$d\in \met(X)$.
\end{cor}

\section{\bf Proofs of the main results}\label{sec:proof}

The following proposition can be found in 
 \cite[Proposition 2.1]{Ishiki2022comeager}
 (See also \cite[Proposition 3.1]{ishiki2021dense}). 
\begin{prop}\label{prop:amal}
Let $(X, d)$ be a metric space. 
Let $I$ be a set and 
 $\{B_i\}_{i\in I}$  a covering  of $X$  
consisting of  mutually disjoint  clopen  subsets. 
Let $P=\{p_i\}_{i\in I}$ be points with $p_i\in B_i$
and 
$\{e_i\}_{i\in I}$ 
a set of metrics such that 
$e_i\in \met(B_i)$. 
Let $h$ be a metric on $P$
generating the discrete topology on $P$.  
We define a function $D:X^2\to [0, \infty)$ by 
\[
D(x, y)=
\begin{cases}
e_i(x, y) & \text{ if $x, y\in B_i$;}\\
e_i(x, p_i)+h(p_i, p_j)+e_j(p_j, y)  & \text{if $x\in B_i$ and $y\in B_j$.}
\end{cases}
\]
Then 
$D\in \met(X)$ 
and 
$D|_{B_i^2}=e_i$ 
for all $i\in I$. 
Moreover, 
if for every $i\in I$ 
we have 
$\yodiam_{d}(B_i)\le \epsilon$
 and 
 $\yodiam_{e_i}(B_i)\le \epsilon$, 
 then 
$\metdis_{X}(D, d)\le 4\epsilon+\metdis_{P}(d|_{P^2}, h)$. 
\end{prop}

\begin{prop}\label{prop:amalcomp}
Under the same assumptions as  in 
Proposition \ref{prop:amal}, 
if $h$ is uniformly discrete and 
each $e_{i}$ is complete, 
then 
the metric $D$ is complete. 
\end{prop}
\begin{proof}
Take $c\in (0, \infty)$ such that 
$c<h(p_{i}, p_{j})$ for all distinct $i, j\in I$. 
By the definition of $D$, 
if $x\in B_{i}$ and $y\in B_{j}$ and $i\neq j$, 
we have $c<D(x, y)$. 
Take a Cauchy sequence 
$\{x_{n}\}_{n\in \zz_{\ge 0}}$ 
in 
$(X, D)$ and 
take a sufficient large number $N\in \zz_{\ge 0}$ such that 
for all $n, m>N$, we have $D(x_{n}, x_{m})<c$. 
Then we observe that 
there exists $i\in I$ satisfying that 
$\{\, x_{n}\mid N<n\, \}\subset B_{i}$. 
Since $D|_{B_{i}^{2}}=e_{i}$ and $e_{i}$ is complete, 
the sequence $\{x_{i}\}_{i\in \zz_{\ge 0}}$ has a limit point. 
Therefore we conclude that 
$(X, D)$ is complete. 
\end{proof}

We now prove 
Theorem \ref{thm:main1}. 
\begin{proof}[Proof of Theorem \ref{thm:main1}]
Let $X$ be a strongly $0$-dimensional 
metrizable space
with $\card(X)\le \yocontinuum$. 
Let $d\in \met(X)$
and  $\epsilon \in (0, \infty)$. 

Put $\eta=\epsilon/5$. 
Take $k\in \zz_{\ge 0}$ such that 
$2^{-k}\le \eta$. 
Since $X$ is paracompact and strongly 
$0$-dimensional, 
we can 
take a mutually disjoint open cover 
$\{O_{\alpha}\}_{\alpha<\tau}$ of $X$ with 
$\yodiam_{d}(O_{\alpha})\le \epsilon$, 
where $\tau\le \yocontinuum$
(see \cite[Proposition 1.2 and Corollary 1.4]{ellis1970extending}). 

For each $\alpha< \tau$, 
we take $p_{\alpha}\in O_{\alpha}$. 
Put $P=\{\, p_{\alpha}\mid \alpha\in \tau\, \}$. 
Then $P$ is a discrete space and 
$d|_{P^{2}}$ 
is a discrete metric on $P$. 
 Lemma \ref{lem:approxEk}
 guarantees the existence of 
 a strongly rigid 
uniformly discrete metric 
$h\in \met(P; \yoxset_{k})$ with 
$\metdis_{P}(d_{P^{2}}, h)\le \eta$. 
Applying 
Lemma \ref{lem:smallFk} to 
$O_{i}$, 
we obtain 
a strongly rigid metric 
$e_{\alpha}\in \met(O_{\alpha}; \yoccset_{k, \alpha})$ with 
$\yodiam_{e_{\alpha}}(O_{\alpha})\le 2^{-k}\le \eta$. 
We define a metric $e$ by 
\[
e(x, y)=
\begin{cases}
e_{\alpha}(x, y) & \text{ if $x, y\in B_{\alpha}$;}\\
e_{\alpha}(x, p_{\alpha})+
h(p_{\alpha}, p_{\beta})+e_{\beta}(p_{\beta}, y)  & \text{if $x\in B_{\alpha}$ and $y\in B_{\beta}$.}
\end{cases}
\]
Applying Proposition \ref{prop:amal} to 
$\{O_{\alpha}\}_{\alpha<\tau}$, $P$, 
$\{e_{\alpha}\}_{\alpha< \tau}$, 
$h$, and $\eta$, 
we obtain $e\in \met(X)$ and 
$\metdis(d, e)\le 5\eta=\epsilon$. 

We shall prove $e\in \yoipt(X)$. 
From  the definition of $e$, and 
Lemma \ref{lem:xyzabc}, 
it follows  that 
for all $x, y, u, v\in X$ with 
$x\neq y$, $u\neq v$, 
and $\{x, y\}\neq \{u, v\}$, 
the numbers 
$e(x, y)$ and $e(u, v)$ are
linearly independent over $\qq$. 
Hence we conclude that $e\in \yoipt(X)$. 

To prove the latter part, 
we 
assume that $X$ is completely metrizable. 
Since each $O_{\alpha}$ is clopen in 
$X$, 
the set  $O_{\alpha}$ is completely metrizable. 
By  the latter part of Lemma \ref{lem:smallFk}, 
we can choose each $e_{\alpha}$ as a 
complete metric. 
Since $h$ is uniformly discrete, 
Proposition \ref{prop:amalcomp}
implies that $e$ is a complete metric. 
This finishes  the proof of Theorem \ref{thm:main1}. 
\end{proof}

The proof of the following proposition is 
analogous with 
that  of \cite[Theorem 2]{MR2831899}. 
\begin{prop}\label{prop:srgdelta}
If 
 $X$ is  a strongly $0$-dimensional 
$\sigma$-compact
metrizable space,
then the set 
$\yosr(X)$ is $G_{\delta}$ 
in $\met(X)$. 
\end{prop}
\begin{proof}
Take a sequence
$\{K_{i}\}_{i\in \zz_{\ge 0}}$
 of compact subsets of $X$ such that 
$X=\bigcup_{i\in \zz_{\ge 0}}K_{i}$ and 
$K_{i}\subset K_{i+1}$ for all $i\in \zz_{\ge 0}$. 
For $n, m \in \zz_{\ge 0}$, 
we denote by  $L_{n, m}$
the set of all 
 $d\in \met(X)$ such that there exists $x, y, u, v\in K_{n}$  with 
 \begin{enumerate}
 \item 
 $d(x, y)=d(u, v)\ge 2^{-m}$;
 \item $d(x, u)+d(y, v)\ge 2^{-m}$;
 \item $d(x, v)+d(u, y)\ge 2^{-m}$. 
 \end{enumerate}
 We now show that $L_{n, m}$ is a closed subset of 
 $\met(X)$. 
 Take  a sequence $\{e_{i}\}_{i\in \zz_{\ge 0}}$ in 
 $L_{n, m}$ and
 take $d\in \met(X)$
 such that  
 $e_{i}\to  d$
 as $i\to \infty$. 
 We shall show $d\in L_{n, m}$. 
 By extracting a subsequence using 
 the compactness of $K_{n}$ if necessary, 
 we may assume that there exist 
  sequences 
  $\{x_{i}\}_{i\in \zz_{\ge 0}}$,
   $\{y_{i}\}_{i\in \zz_{\ge 0}}$, 
  $\{u_{i}\}_{i\in \zz_{\ge 0}}$, 
  and 
  $\{v_{i}\}_{i\in \zz_{\ge 0}}$ in 
  $K_{n}$, 
  and 
    points $x, y, z, w\in K_{n}$, such that 
\begin{enumerate}
\item  $e_{i}(x_{i}, y_{i})=e_{i}(u_{i}, v_{i})\ge 2^{-m}$
 for all $i\in \zz_{\ge 0}$; 
\item  $e_{i}(x_{i}, u_{i})+e_{i}(y_{i}, v_{i})\ge 2^{-m}$ for all $i\in \zz_{\ge 0}$; 
 \item $e_{i}(x_{i}, v_{i})+e_{i}(y_{i}, u_{i})\ge 2^{-m}$ for all $i\in \zz_{\ge 0}$; 
 \item $x_{i}\to x$, $y_{i}\to y$, 
  $u_{i}\to u$, and 
  $v_{i}\to v$ as $i\to \infty$. 
\end{enumerate}

Since $d$ and $e_{i}$ generate the same topology of 
$X$ and since $e_{i}\to d$ as $i\to \infty$, 
 if $p\in \{x, y, u, v\}$ and $q\in \{x, y, u, v\}$, 
then we have 
$e_{i}(p_{i}, q_{i})\to d(p, q)$ as 
$i\to \infty$. 
Thus, we obtain 
$d(x, y)=d(u, v)\ge 2^{-m}$ and 
$d(x, u)+d(y, v)\ge 2^{-m}$
and $d(x, v)+d(y, u)\ge 2^{-m}$. 
Therefore $d\in L_{n, m}$, 
and hence $L_{n, m}$ is closed. 
 Put $G_{n, m}=\met(X)\setminus L_{n, m}$. 
 Then 
 each $G_{n, m}$ is open in $\met(X)$,  and 
 we obtain 
 \[
 \yosr(X)=\bigcap_{n, m\in \zz_{\ge 0}}G_{n, m}.
 \]
 This proves  the proposition. 
\end{proof}

\begin{proof}[Proof of Theorem \ref{thm:main15}]
Let $X$ be a strongly $0$-dimensional 
metrizable space
with $\card(X)\le \yocontinuum$. 

Due to  Proposition \ref{prop:srgdelta}, 
we only need to prove that 
$\yosr(X)$ is dense in $\met(X)$. 
We now show $\yoipt(X)\subset \yosr(X)$. 
Take $d\in \yoipt(X)$. 
Take $x, y, u, v\in X$ with $x\neq y$
and $u\neq v$, and $\{x, y\}\neq \{u, v\}$. 
Then $d(x, y)$ and $d(u, v)$ are linearly independent 
over $\qq$. In particular, 
we obtain 
$d(x, y)\neq d(u, v)$, and hence 
$d\in \yosr(X)$. 
Thus, we have $\yoipt(X)\subset \yosr(X)$. 
According to  Theorem \ref{thm:main1}, 
we observe that $\yosr(X)$ is dense in $\met(X)$. 
This completes the proof of 
Theorem \ref{thm:main15}. 
\end{proof}

\begin{proof}[Proof of Theorem \ref{thm:main2}]
Let $X$ be a strongly $0$-dimensional 
metrizable space. 
Assume that  $X$ is $\sigma$-compact and satisfies 
$3\le \card(X)\le \yocontinuum$. 

We only need to prove $\yosr(X)\subset \yorrr(X)$.
Take  $d\in \yosr(X)$, 
and let $f\colon (X, d)\to (X, d)$ be a bijective isometry. 
Take arbitrary $x\in X$, and 
take two  points $y, z$ in $X$ with 
$\card(\{x, y, z\})=3$. 
Since $d(x, y)=d(f(x), f(y))$ and $d(x, z)=d(f(x), f(z))$, 
we have $\{x, y\}=\{f(x), f(y)\}$ and $\{x, z\}=\{f(x), f(z)\}$.
Then we obtain 
$f(x)\in \{x, y\}\cap \{x, z\}=\{x\}$, 
and hence $f(x)=x$. 
Since $x$ is arbitrary, we conclude that $f$ is the identity map, 
which implies that $\yosr(X)\subset \yorrr(X)$. 
This finishes the proof of Theorem \ref{thm:main2}. 
\end{proof}

Before proving Theorem \ref{thm:plus1}, 
we introduce some notions. 
Let $X$ and $Y$ be topological spaces. 
A continuous  map $f\colon X\to Y$ is 
said to be \empty{proper} if for every compact subset of $Y$, the set $f^{-1}(K)$ is compact. 
For a metric space $(X, d)$, the metric $d$ is 
\emph{proper} if all closed balls in $(X, d)$ are compact. 
Note that $d$ is proper if and only if 
a map $x\mapsto d(x, p)$ is a proper map for all $p\in X$. We denote by $B(x, r; d)$ the closed ball 
of $(X, d)$ centered at $x$ with radius $r$. 
\begin{proof}[Proof of Theorem \ref{thm:plus1}]
Let $X$ be a strongly $0$-dimensional 
second-countable
 locally compact Hausdorff space. 
Observe  that $X$ is metrizable. 

Let $\yoppset$ be the set of all 
proper metrics in $\met(X)$. 
Since $X$ is second-countable and  locally compact, 
we have  $\yoppset\neq \emptyset$ 
(see for example \cite{ishiki2022prop}). 
We now show that $\yoppset$ is open in 
$\met(X)$. 
Take $d\in \yoppset$ and take 
$e\in \met(X)$ such that $\metdis_{X}(d, e)\le 1$. 
In this setting, we notice that 
$B(x, r; e)\subset B(x, r+1; d)$ for all $x\in X$ and $r\in (0, \infty)$. 
Since $d$ is proper, so is $e$. 
Thus, the set $\yoppset$ is open in $\met(X)$. 

Due to Theorem \ref{thm:main15}, 
the set $\yosr(X)$ is dense in $\met(X)$. 
Since $\yoppset$ is non-empty and open, 
we obtain $\yoppset \cap \yosr(X)\neq \emptyset$, 
and we can take a member from this set, say $d$. 
Recall that for each $\xi\in X$, the map $F_{\xi}\colon X\to [0, \infty)$ is defined as $F_{\xi}(x)=d(x, \xi)$. 
Fix $\xi\in X$. 
Next we verify that $F_{\xi}$ is a topological 
embedding. 
Since $d$ is strongly rigid, the map $F_{\xi}$ is 
injective. 
According to $d\in \yoppset$, the map $F_{\xi}$ is proper as a map. 
From the fact that every proper map into a  metrizable space is  a closed map (see for instance  \cite{MR254818}), 
it follows that $F_{\xi}$ is a closed map. 
Therefore we conclude that $F_{\xi}$ is 
a topological embedding. 
This finishes the proof of Theorem \ref{thm:plus1}. 
\end{proof}

\begin{ac}
The author would like to thank the referee
for helpful comments and suggestions. 
\end{ac}


\bibliographystyle{plain}
\bibliography{bibtex/disco.bib}

\end{document}